\documentclass[12pt]{amsart}
\usepackage{amsmath,amssymb,amsthm,tikz}
\usetikzlibrary{arrows,decorations.markings}
\usepackage[letterpaper,margin=1in]{geometry}
\usepackage{comment}
\usepackage[centertableaux]{ytableau}

\usepackage[colorlinks=true, pdfstartview=FitV, linkcolor=blue, citecolor=blue, urlcolor=blue]{hyperref}

\newtheorem{theorem}{\textbf{Theorem}}[section]
\newtheorem{proposition}[theorem]{\textbf{Proposition}}
\newtheorem{corollary}[theorem]{\textbf{Corollary}}

\newtheorem{lemma}[theorem]{\textbf{Lemma}}

\theoremstyle{definition}
\newtheorem{remark}[theorem]{Remark}

\theoremstyle{remark}
\newtheorem{example}[theorem]{Example}

\newenvironment{subproof}[1][\proofname]{%
  \begin{proof}[#1]%
}{%
  \end{proof}%
}

\newcommand{\truearrow}{[line width=.25mm,decoration={markings,mark=at position .99 with {\arrow[scale=3,>=stealth]{>}}},postaction={decorate}]}

\newcommand{\ket}[1]{\lvert #1 \rangle}
\newcommand{\bra}[1]{\langle #1 \rvert}
\newcommand{\braket}[2]{\langle #1 | #2 \rangle}
\newcommand{\abs}[1]{\lvert #1 \rvert}
\newcommand{\normord}[1]{ {: \mathrel{#1} :} }  

\newcommand{\dv}{|\!|}
\newcommand{\FPS}[1]{[\![#1]\!]}  
\newcommand{\FLS}[1]{(\!(#1)\!)}  
\newcommand{\one}{\mathbf{1}}

\newcommand{\fermionfock}{\mathfrak{F}}

\newcommand{\ii}{\mathbf{i}}

\newcommand{\pp}{\mathbf{p}}

\newcommand{\xx}{\mathbf{x}}
\newcommand{\yy}{\mathbf{y}}

\newcommand{\bal}{\boldsymbol{\alpha}}
\newcommand{\bbe}{\boldsymbol{\beta}}
\newcommand{\bbb}{\mathsf{b}}

\newcommand{\gl}{\mathfrak{gl}}
\newcommand{\mfa}{\mathfrak{a}}

\newcommand{\mcC}{\mathcal{C}}

\newcommand{\mcP}{\mathcal{P}}

\newcommand{\ZZ}{\mathbb{Z}}

\newcommand{\CC}{\mathbb{C}}
\newcommand{\End}{\operatorname{End}}

\renewcommand{\dotsc}{\cdots}  

\definecolor{darkred}{rgb}{0.7,0,0} 
\newcommand{\defn}[1]{{\color{darkred}\emph{#1}}} 

\usepackage[colorinlistoftodos]{todonotes}

\begin{document}

\title{On the Boson-Fermion Correspondence for Factorial Schur Functions}

\author{Daniel Bump}
\address[D.~Bump]{Department of Mathematics, Stanford University, Stanford, CA 94305-2125}
\email{bump@math.stanford.edu}
\urladdr{https://math.stanford.edu/~bump/}

\author{Andrew Hardt}
\address[A.~Hardt]{Department of Mathematics, University of Illinois Urbana-Champaign, Urbana, IL 61801}
\email{ahardt@illinois.edu}
\urladdr{https://andyhardt.github.io/}

\author{Travis Scrimshaw}
\address[T.~Scrimshaw]{Department of Mathematics, Hokkaido University, 5 Ch\=ome Kita 8 J\=onishi, Kita Ward, Sapporo, Hokkaid\=o 060-0808}
\email{tcscrims@gmail.com}
\urladdr{https://tscrim.github.io/}

\maketitle

\begin{abstract}
We give an algebraic (non-analytic) proof of the deformed boson-fermion Fock space construction of Molev's double supersymmetric Schur functions, among other results, from our previous paper.
In other words, we make no assumptions on the variables and parameters.
By specializing to a finite number of variables and shifting parameters, we recover the factorial Schur functions.
Furthermore, we realize the bosonic construction through a representation of a completion of the infinite rank general linear Lie algebra.
\end{abstract}


\section{Introduction}

This is a companion to our other paper~\cite{BHS}; here, we give purely algebraic proofs when one set of parameters is zero.
In more detail, we had to impose certain analytic conditions in~\cite{BHS} on the two sets of parameters $\bal = (\alpha_i)_{i \in \ZZ}$ and $\bbe = (\beta_i)_{i\in\ZZ}$ (along with the auxiliary indeterminates).
Yet, if we take $\bbe = 0$ (that is, $\beta_i = 0$ for all $i \in \ZZ$), then we can prove our results without imposing any conditions on $\bal$ (or on the supersymmetric function variables $\xx$ and $\yy$) by using Laurent series rings such as $\ZZ[\bal]\FLS{z}$.
This paper provides the necessary algebraic proofs to results that previously had analytic assumptions.
In other words, we can work over the coefficient rings $\ZZ[\bal]$ instead of $\CC$ (with $\bal, \bbe \subseteq \CC$).
All of the other results from~\cite{BHS} will follow without modification; in particular, we have the following results (taking $\bbe = 0$).

\begin{theorem}
\label{thm:boson_fermion}
Let $\lambda$ be a partition.
There exists an $\bal$ deformation of the classical boson-fermion correspondence such that the image of the basis vector $\ket{\lambda}$ of fermionic Fock space is Molev's double supersymmetric Schur function~\cite{MolevFactorialSupersymmetric,Molev09} of shape $\lambda$ with the parameters reindexed by $\alpha_i\mapsto \alpha_{1-i}$.
Moreover, the image of $\ket{\lambda}$ under a natural dual version equals Molev's dual Schur functions.
\end{theorem}

\begin{theorem}
\label{thm:vertex_model_equivalence}
The deformed half vertex operators correspond to row transfer matrices of known natural solvable five vertex lattice models with row and column parameters.
\end{theorem}

We refer to~\cite[Thm.~5.1, Thm.~6.10, Cor.~6.11]{BHS} (with taking $\bbe = 0$) for precise statements and an introduction to this topic, including historical references.
Other than the proofs of Theorem~\ref{thm:boson_fermion} and Theorem~\ref{thm:vertex_model_equivalence} (which would be copied verbatim), this paper is written to be self-contained and have numerous examples.

One key advantage of taking $\bbe = 0$, beyond allowing formal/algebraic proofs, is that we perform computations that range over finite sums.
Therefore, the complexity of working with the constructions and corresponding symmetric functions is reduced.
Additionally, it allows us to give a representation theoretic interpretation (without any analytic conditions) of our deformed boson-fermion construction.
Indeed, in Section~\ref{sec:lie_algebra}, we work in a completion of $\gl_{\infty}$ that we call near (upper) triangular matrices, which has an associative algebra structure.
The Lie algebra of this completion has appeared implicitly though the corresponding Lie group in~\cite[Ch.~14]{KacInfinite}.
From this point we follow the classical $\bal = 0$ construction given in~\cite[Ch.~14]{KacInfinite} (see also \cite[Ch,~4--6]{KacRaina}) by taking a central extension to yield the Heisenberg algebra.
Two key computations in our proof are identities of symmetric functions in different sets of variables (Proposition~\ref{prop:cocyclecase} and Proposition~\ref{prop:heidentity}).

In Section~\ref{sec:double_funcs}, we recall some constructions of Molev's double supersymmetric Schur functions~\cite{MolevFactorialSupersymmetric,Molev09}.
Using the generating series description of the double elementary and homogeneous supersymmetric functions and some basic identities of shifted powers from Section~\ref{sec:prelims}, we obtain a new proof of the action of the involution $\omega$ that interchanges $h_k(\xx/\yy) \leftrightarrow e_k(\xx/\yy)$ on the double supersymmetric functions (Proposition~\ref{prop:omega};~\cite[Cor.~5.14]{BHS}).

In Section~\ref{sec:fermion_fields}, we give the main result, a fully algebraic proof of the key result~\cite[Prop.~4.1]{BHS} (Theorem~\ref{thm:main_theorem}) that yields the deformed boson-fermion correspondence.
We then detail a number of additional changes from~\cite{BHS} in this setting.
In particular, we can see that the classical supersymmetric functions are finite sums of double Schur functions of descending degree (i.e.\ lower filtered).
These in turn are finite sums in the usual supersymmetric function bases.
Therefore, the product of any two double supersymmetric functions is a finite sum, which includes the Murnagham--Nakayama rule~\cite[Thm~5.23]{BHS} and the product of two double Schur functions.

In~\cite[Sec.~5.4]{BHS}, we stated that the raising operator formula from~\cite{Fun12} does not match the factored contour integral formula from~\cite[Eq.~(66)]{BHS} (but does reduce to the classical proof from, \textit{e.g.},~\cite{Baker96} when $\bal = \bbe = 0$).
In Section~\ref{sec:raising_operators}, we give a more detailed analysis of the differences between the formulas when $\bbe = 0$.

Lastly, we give a detailed analysis of the (skew) Pieri rule~\cite[Cor.~6.15]{BHS} in Section~\ref{sec:skew_pieri} and perform a comparison with~\cite[Prop.~3.4]{Fun12}.
While we are unable to show the Graham positivity, we can show in the (straight shape) Pieri rule that our formulas do not have any monomial cancellations as opposed to~\cite{Fun12}. 

On the other hand, a cost of the $\bbe = 0$ specialization, beyond the loss of generality, is that it conceals a symmetry between the $\bal$ and $\bbe$ variables.
Yet, we conclude the introduction by noting that if we instead set $\bal = 0$ instead of $\bbe$, we end up working with the dual objects, and thus we obtain equivalent statements by~\cite[Sec.~4.4, Thm.~5.29]{BHS}.
This is essentially the same as applying the natural Clifford algebra adjoint/dual.
As such, we also obtain analogous algebraic results for the dual Schur functions.

\subsection*{Acknowledgements}

The authors thank Slava Naprienko for numerous invaluable conversations for this paper and for our previous paper~\cite{BHS}.
This work benefited from computations performed using \textsc{SageMath}~\cite{sage}.

T.S.~was partially supported by Grant-in-Aid for JSPS Fellows 21F51028 and for Scientific Research for Early-Career Scientists 23K12983.
A.H.~was partially supported by NSF RTG grant DMS-1937241.

\section{Preliminaries}
\label{sec:prelims}

To make this paper self-contained, we will set the necessary notation and give the requisite definitions.
We note that our notation will match~\cite{BHS} except we will use the shorthand omitting the $\bbe = 0$ parameters, such as $J_k^{(\bal)} = J_k^{(\bal; 0)}$.

Let $\bal = (\cdots, \alpha_{-1}, \alpha_0, \alpha_1, \cdots)$ be a set of commuting parameters indexed by $\ZZ$.
Let $\xx = (x_1, x_2, \cdots)$ and $\yy = (y_1, y_2, \cdots)$ be commuting indeterminants, and let $\xx_n = (x_1, \dotsc, x_n)$ and $\yy_n = (y_1, \dotsc, y_n)$ formed by setting $x_i = y_i = 0$ for all $i > n$.
Let $\sigma$ and $\iota$ be automorphisms of $\ZZ[\bal]$ defined by $\sigma \alpha_i \mapsto \alpha_{i+1}$ and $\iota \alpha_i \mapsto \alpha_{1-i}$, respectively.
We will often consider these as acting only on $\bal$; \textit{e.g.}, $\sigma \bal$.

\begin{remark}
\label{rem:formal_contours}
In order to transfer our algebraic statements into analytic statements, we only require the simple condition that $\sup \{ \abs{\alpha_i} \mid i \in \ZZ\} < \infty$.
This will allow us to change our formal contour integrals $\oint f(z) \, \frac{dz}{2\pi\ii} = f_{-1}$, where $f(z) = \sum_{i=k}^{\infty} f_i z^i \in \ZZ\FLS{z}$, into actual contour integrals $\oint_{\eta} f(z) \, \frac{dz}{2\pi\ii}$ with the contour $\eta$ being a small counterclockwise circle around zero.
In particular, the circle will have radius $r < \abs{\alpha_i^{-1}}$ for all $i \in \ZZ$ (by convention, $\abs{0^{-1}} = \infty$).
Indeed, the poles of our functions $f$ will occur at $\alpha_i^{-1}$ for certain $i \in \ZZ$ and at $0$ unless otherwise stated.
\end{remark}

The (classical) \defn{elementary symmetric functions} and \defn{homogeneous symmetric functions} are
\[
e_k(\xx_n) = \sum_{1 \leqslant i_1 < \cdots < i_k \leqslant n} x_{i_1} \cdots x_{i_k},
\qquad\qquad
h_k(\xx_n) = \sum_{1 \leqslant i_1 \leqslant \cdots \leqslant i_k \leqslant n} x_{i_1} \cdots x_{i_k},
\]
respectively.
For brevity, we will use the following notation
\[
e_k(-\bal_{(i,j)}) = e_k(-\alpha_{i+1}, \dotsc, -\alpha_{j-1}),
\qquad\qquad
h_k(\bal_{[i,j]}) = h_k(\alpha_i, \dotsc, \alpha_j).
\]

\subsection{Shifted powers}
\label{sec:shifted_powers}

The \defn{shifted powers} are defined as
\[
(z^{-1}|\bal)^k := \prod_{i=k+1}^0 (z^{-1} - \alpha_i)^{-1} \prod_{i=1}^k (z^{-1} - \alpha_i),
\]
and we note that at most one of these products is not $1$.
We have defined the shifted powers in terms of $z^{-1}$ in order to have them belong to $\ZZ[\bal]\FLS{z}$.
Like in~\cite{BHS}, we will never use $(z^{-1}|\bal)^k$ to denote the $k$-fold product of $(z^{-1}|\bal)$ with itself, so there will be no danger of confusion.
The set $\{(z^{-1}|\bal)^k \mid k \in \ZZ\}$ forms a basis for $\ZZ[\bal]\FLS{z}$ by triangularity as $(z^{-1}|\bal)^k$ has valuation $-k$ (under the standard valuation of Laurent polynomials/series).

As a consequence, we have a recursive algorithm for expressing any formal Laurent series in a shifted power basis.
However, we will find it useful to have an explicit expression for $z^k$ in terms of shifted powers.
To do so, begin by noting two useful relations:
\begin{subequations}
\label{eq:z_times_shifted}
\begin{align}
z^{-1} (z^{-1}|\bal)^k & = (z^{-1}|\bal)^{k+1} + \alpha_{k+1} (z^{-1}|\bal)^k, \label{eq:z_times_zs}
\\
z (z^{-1}|\bal)^k &= \sum_{m=0}^{\infty} (-1)^m \alpha_k \cdots \alpha_{k-m+1} (z^{-1}|\bal)^{k-m-1}. \label{eq:zinv_times_zs}
\end{align}
\end{subequations}
Equation~\eqref{eq:z_times_zs} follows from a direct computation, and~\eqref{eq:zinv_times_zs} is given by multiplying~\eqref{eq:z_times_zs} by $z$ on both sides and iterating the result. 

\begin{example}
We have
\begin{align*}
z^{-1} (z^{-1}|\bal)^{-3} & = \frac{z^{-1}}{(z^{-1}-\alpha_0)(z^{-1}-\alpha_{-1})(z^{-1}-\alpha_{-2})}
\\ & = \frac{1}{(z^{-1}-\alpha_0)(z^{-1}-\alpha_{-1})} + \frac{\alpha_{-2}}{(z^{-1}-\alpha_0)(z^{-1}-\alpha_{-1})(z^{-1}-\alpha_{-2})}
\\ & = (z^{-1}|\bal)^{-2} + \alpha_{-2} (z^{-1}|\bal)^{-3},
\allowdisplaybreaks \\
z (z^{-1}|\bal)^2 & = (z^{-1}-\alpha_1)(1 - \alpha_2 z) = (z^{-1}|\bal)^1 - \alpha_2(1 - \alpha_1 z)
\\ & = (z^{-1}|\bal)^1 - \alpha_2 (z^{-1}|\bal)^0 + \alpha_2 \alpha_1 z (z^{-1}|\bal)^0,
\end{align*}
\end{example}

\begin{proposition}
\label{prop:shifted_cob}
For $m > 0$, we have
\begin{align*}
	z^{-m} &= \sum_{k=0}^{m}h_{m-k}(\bal_{[1,k+1]}) (z^{-1}|\bal)^k, &
	(z^{-1}|\bal)^m &= \sum_{k=0}^m e_{m-k}(-\bal_{(0,m+1)}) z^{-k}, \\
	z^m &= \sum_{k=m}^{\infty} e_{k-m}(-\bal_{(1-k,1)}) (z^{-1} | \bal)^{-k}, &
	(z^{-1}|\bal)^{-m} &= \sum_{k=m}^{\infty}h_{k-m}(\bal_{[1-m,0]}) z ^k.
\end{align*}
\end{proposition}

\begin{proof}
For $m > 0$, the Laurent expansion of $(z^{-1}|\bal)^m$ is classical, and furthermore
\[
(z^{-1}|\bal)^{-m} = \prod_{i=0}^{m-1}\frac{1}{z^{-1}-\alpha_{-i}} = z^m \prod_{i=0}^{m-1}\frac{1}{1-\alpha_{-i} z} = z^m \sum_{k=0}^{\infty} h_k(\bal_{[1-m,0]}) z^k.
\]
Thus we just need to consider the expansion of $z^{-m}$ in terms of the shifted powers.
For $m > 0$, by applying~\eqref{eq:z_times_zs} in a straightforward induction argument, we obtain
\begin{align*}
z^{-m} & = z^{1-m}\bigl( (z^{-1}|\bal) + \alpha_1 \bigr) = z^{2-m}\bigl( (z^{-1}|\bal)^2 + (\alpha_1 + \alpha_2)(z^{-1}|\bal) + \alpha_1^2 \bigr)
\\ & = z^{3-m}\bigl( (z^{-1}|\bal)^3 + h_1(\alpha_1, \dotsc, \alpha_3) (z^{-1}|\bal)^2 + h_2(\alpha_1, \alpha_2) (z^{-1}|\bal) + \alpha_1^3 \bigr)
\\ & = \cdots
\\ & = \sum_{k=0}^m h_k(\bal_{[1,m-k+1]}) (z^{-1}|\bal)^{k-m} = \sum_{k=0}^m h_{m-k}(\bal_{[1,k+1]})(z^{-1}|\bal)^k.
\end{align*}
The proof for $m < 0$ is similar.
\end{proof}

Additionally, we will use the fact that $\{ \frac{1}{(z|\bal)^k} \mid k \in \ZZ \}$ is another basis of $\CC\FLS{z}$, which can be proven by triangularity or by using
\begin{equation}
\label{eq:shifted_inversion}
\frac{1}{(z^{-1}|\bal)^k} = (z^{-1}|\iota \bal)^{-k} = (z^{-1}|\sigma^k\bal)^{-k}.
\end{equation}
Note that~\eqref{eq:shifted_inversion} also allows us to compute explicit formulas from Proposition~\ref{prop:shifted_cob}.
The triangularity follows by noting that the valuation of $\frac{1}{(z^{-1}|\bal)^k}$ is $k$.

Next, we compute how the shift $\sigma$ acts on these bases.

\begin{proposition}
\label{prop:shifted_action}
We have
\begin{align*}
(z^{-1}|\sigma^{-1}\bal)^k & = (z^{-1}|\bal)^k + (\alpha_k - \alpha_0) (z^{-1}|\bal)^{k-1},
\\
\frac{1}{(z^{-1}|\sigma\bal)^k} & = \frac{1}{(z^{-1}|\bal)^k} + \frac{\alpha_{k+1} - \alpha_1}{(z^{-1}|\bal)^{k+1}}.
\end{align*}
\end{proposition}

\begin{proof}
By direct computation.
For the first equality, factor our the $(z^{-1}|\bal)^{k-1}$ from the right hand side.
For the second, combine the two fractions on the right hand side.
\end{proof}

We can define inner products on $\CC\FLS{z}$ such that these shifted power bases are orthonormal by using the following result.

\begin{proposition}[{\cite[Prop.~2.3]{BHS}}]
\label{prop:orthonormality}
\;
\[
\oint \frac{z^{-1} (z^{-1}|\bal)^{n-1}}{(z^{-1}|\bal)^k} \frac{dz}{2\pi\ii z}
= \oint (z^{-1}|\sigma^k \bal)^{n-k-1} \frac{dz}{2\pi\ii z^2}
= \delta_{nk}.
\]
\end{proposition}

The next two results generalize Proposition~\ref{prop:orthonormality}. 

\begin{proposition}
\label{prop:orthonormality_gen}
For any $k \geqslant 0$ and $(i_1, \dotsc, i_k) \in \ZZ^k$, we have
\[
\oint \prod_{j=1}^k (z^{-1} - \alpha_{i_j}) \frac{dz}{2\pi\ii z^2} = 0.
\]
\end{proposition}

\begin{proof}
This follows from the fact that the degree of the Laurent polynomial is $-2$, and so there is no contribution to the (unique) residue at $0$.
\end{proof}

\begin{proposition}
\label{prop:orthonormality_gen2}
For any $k > k' \geqslant 0$ and $(i_1, \dotsc, i_k) \in \ZZ^k$ and $(i'_1, \dotsc, i'_{k'}) \in \ZZ^{k'}$, we have
\[
\oint \frac{\prod_{j=1}^{k'} (z^{-1} - \alpha_{i_j})} {\prod_{j=1}^k (z^{-1} - \alpha_{i_j})} \frac{dz}{2\pi\ii z^2} = \delta_{k,k'+1}.
\]
\end{proposition}

\begin{proof}
Note that the valuation of $(z^{-1} - \alpha_j)^{-1} = \frac{z}{(1 - \alpha_j z)}$ is $1$, and so if $k - k' > 1$, then resulting valuation is nonnegative.
If $k = k'+1$, then the valuation is $-1$ and it is easy to see the coefficient of $z^{-1}$ is $1$.
\end{proof}

\subsection{Symmetric function identities}

We prove some identities of symmetric functions that are variants of
the classical identity $\sum_{i=0}^n (-1)^i e_i(\xx) h_{n-i}(\xx) = \delta_{n0}$ as the input parameters now vary.
They will play a pivotal role in what follows.
There is also a $\bbe$ version of this that can be deduced from~\cite[Thm~4.4, Prop.~4.6]{BHS}.

\begin{proposition}
\label{prop:cocyclecase}
If $0 < k \leqslant \ell$, then
\begin{equation}
\label{eq:cocyclecase}
   \sum_{\substack{1 \leqslant j \leqslant \ell \\ j - \ell \leqslant i \leqslant \min (0, j - k)}} h_{i - j + \ell}(\bal_{[i,j]}) \, e_{j - i - k} (-\bal_{(i,j)})
   = k \delta_{k\ell}.
\end{equation}
\end{proposition}

\begin{proof}
In the case $k = \ell$, the conditions imply that $i = j - \ell = j - k$ and every term equals~$1$.
We leave the details of this case to the reader and assume that $\ell > k$.
In that case, the right-hand side is a homogeneous polynomial of degree $l - k$, and we will show that when it is expanded in monomials, all terms cancel.

Let
\[
T = \{ (i, j) |1 \leqslant j \leqslant \ell, j - \ell \leqslant i \leqslant \min (0, j - k) \}.
\]
Given $(i, j) \in T$, define $\Omega(i, j)$ to be the set of all pairs $(C, D)$, where $C = \{c_1, \ldots, c_{i - j + \ell} \}$ is a multiset and $D = \{d_1, \ldots, d_{j - i - k} \}$ is a set of integers satisfying
\[
i \leqslant c_1 \leqslant \cdots \leqslant c_{i - j + \ell} \leqslant j,
\qquad
i < d_1 < \cdots < d_{j - i - k} < j.
\]
Note that $\abs{C} + \abs{D} = \ell - k > 0$, so either $C$ or $D$ is nonempty.
By expanding the $h_{i - j + \ell}$ and $e_{j - i - k}$ in~\eqref{eq:cocyclecase}, the left-hand side equals
\begin{equation}
\label{eq:cocycleexpanded} \sum_{(i, j) \in T} (- 1)^{j - i - k} \sum_{(C, D) \in \Omega (i, j)} \alpha^C \alpha^D .
\end{equation}

Let us decompose
\[
\Omega(i, j) = \Omega_C(i, j) \sqcup \Omega_D(i, j) \qquad \left( \text{disjoint} \right),
\]
where $\Omega_C(i, j)$ consists of pairs $(C, D)$ such that either $D = \varnothing$ or $c_1 < d_1$, and $\Omega_D(i, j)$ consists of pairs such that either $C = \varnothing$ or $d_1 \geqslant c_1$.
We will call an element $(C, D)$ of $\Omega_C(i, j)$ \defn{exceptional} if $D = \varnothing$ and $c_1 = \cdots = c_{i - j + \ell} = j$.

We will also need an alternative decomposition
\[
\Omega(i, j) = \Omega^C(i, j) \sqcup \Omega^D(i, j) \qquad \left( \text{disjoint} \right),
\]
where $\Omega^C(i, j)$ consists of pairs where either $D =\varnothing$ or $c_{i - j + l} > d_{j - i + k}$, and $\Omega^D(i, j)$ consists of pairs such that $C = \varnothing$ or $d_{j - i - k} \geqslant c_{i - j + \ell}$.
We will call an element $(C, D)$ of $\Omega^C (i, j)$ \defn{exceptional} if $D = \varnothing$ and $c_1 = \cdots = c_{i - j + \ell} = i$.

We will define operations $\mathcal{L}_C$ on $\Omega_C (i, j)$ and $\mathcal{L}_D$ on $\Omega_D (i, j)$, namely $\mathcal{L}_C$ removes $c_1$ from $C$ and places it at the beginning of $D$, so $\mathcal{L}_C(C, D) = (C', D')$, where
\[
C' = \{ c_2, \dotsc, c_{i - j + l} \}, \qquad D' = \{ c_1, d_1, \dotsc, d_{j - i - k} \}.
\]
Similarly $\mathcal{L}_D$ removes $d_1$ from $D$ and places it at the beginning of $C$.
Let $\mathcal{L}^C$ be the operation on $\Omega^C(i, j)$ that removes $c_{i - j + k}$ from the end of $C$ and places it at the end of $D$, and $\mathcal{L}^D$ removes $d_{j - i - k}$ from the end of $D$ and places it at the end of $C$.

\begin{lemma}
\label{lem:firstinv}Assume that $(i, j) \in T$.
\renewcommand{\theenumi}{\roman{enumi}}
\begin{enumerate}
\item If $(C, D) \in \Omega_C (i, j)$ is not exceptional, then $(i - 1, j) \in T$ and $\mathcal{L}_C (C, D) \in \Omega (i - 1, j)$.

\item If $(C, D) \in \Omega_D (i, j)$ and $i < 0$, then $(i + 1, j) \in T$ and $\mathcal{L}_D (C, D) \in \Omega (i + 1, j)$.

\item If $(C, D) \in \Omega^C (i, j)$ is not exceptional then $(i, j + 1) \in T$ and $\mathcal{L}^C (C, D) \in \Omega (i, j + 1)$.

\item If $(C, D) \in \Omega^D (i, j)$ and $j > 1$, then $(i, j - 1) \in T$ and $\mathcal{L}^D (C, D) \in \Omega (i, j - 1)$.
\end{enumerate}
\end{lemma}

\begin{subproof}
We will consider the first two cases and leave the last two to the reader.
To prove (i), since $C$ is nonempty, we must have $i - j + l > 0$.
This, together with the fact that $(i, j) \in T$ implies that $(i - 1, j) \in T$.
Because $c_1 < d_1$ (or $D$ is empty) we may move $c_1$ to $q$, but have only to check the inequalities $i - 1 < c_1 < j$.
Since $i \leqslant c_1 \leqslant j$, the only way this can fail is that $c_1 = j$.
This implies that $c_1 = \cdots = c_{i - j + l} = j$, and we are in the exceptional case.

To prove (ii), we note that since $D$ is nonempty, $j - i - k > 0$.
Also we are assuming that $i < 0$.
These facts, together with the fact that $(i, j) \in T$ imply that $(i + 1, j) \in T$.
It is easy to see that moving $d_1$ to $C$ gives an element of $\Omega (i + 1, j)$.
\end{subproof}

\begin{lemma}
\label{lem:secondinv}
\mbox{}
\renewcommand{\theenumi}{\roman{enumi}}
\begin{enumerate}
\item Let $(C, D) \in \Omega_C (i, j)$.
  Suppose that the smallest element of $C \cup D$ is $\leqslant 0$.
  Then $(C, D)$ is not exceptional.

\item Let $(C, D) \in \Omega_D (i, j)$.
  Suppose that the smallest element of $C \cup D$ is $\leqslant 0$. Then $i < 0$.

\item Let $(C, D) \in \Omega^C (i, j)$.
  Suppose that the smallest element of $C \cup D$ is $> 0$.
  Then $(C, D)$ is not exceptional.

\item Let $(C, D) \in \Omega^D (i, j)$.
  Suppose that the smallest element of $C \cup D$ is $> 0$.
  Then $j > 1$.
\end{enumerate}
\end{lemma}

\begin{subproof}
For (i), our assumption implies that $c_1 \leqslant 0$, while $1 \leqslant j$ so $c_1 \neq j$, implying that we are not in the exceptional case.
For (ii), we have $i < d_1 \leqslant c_1$ which we are assuming is $\leqslant 0$, so $i < 0$. For (iii), all elements of $C$ are positive, so $c_1 = \cdots = c_{i - j + l} = i$ is ruled out because $i \leqslant 0$.
Finally for (iv), $j = 1$ is impossible since $0 < d_{j - i - k} < j$.
\end{subproof}

We may now define an involution $\eta$ of the disjoint union $\Omega = \bigsqcup_{(i, j) \in T} \Omega (i, j)$ as follows.
Suppose that $(C, D) \in \Omega (i, j)$.
Define
\[
   \eta (C, D) = \begin{cases}
   \mathcal{L}_C (C, D) & \text{if $\min (C, D) \leqslant 0$ and $(C, D) \in \Omega_C (i, j)$,}\\
   \mathcal{L}_D (C, D) & \text{if $\min (C, D) \leqslant 0$ and $(C, D) \in \Omega_D (i, j)$,}\\
   \mathcal{L}^C (C, D) & \text{if $\min (C, D) > 0$ and $(C, D) \in \Omega^C (i, j)$,}\\
   \mathcal{L}_D (C, D) & \text{if $\min (C, D) > 0$ and $(C, D) \in \Omega^D (i, j)$.}
 \end{cases}
\]
By Lemma~\ref{lem:firstinv}, we have $\eta (C, D) \in \Omega$ since Lemma~\ref{lem:secondinv} says the condition on $\min (C, D)$ avoids the problematic cases.
Note that $\eta$ does not change $\min (C, D)$, from which it is easy to see that $\eta$ has order 2.
Terms that correspond under $\eta$ cancel in~\eqref{eq:cocycleexpanded}, and so the sum is zero.
\end{proof}

We remark that our proof of Proposition~\ref{prop:cocyclecase} does not use~\cite[Thm~4.4, Prop.~4.6]{BHS}, but instead is a direct proof.

\begin{proposition}
\label{prop:heidentity}
If $i \leqslant j$, then
\begin{equation}
\label{eq:heidentity}
\sum_{t = i}^j h_{t - i}(\bal_{[t-k,i]}) \, e_{j - t}(-\bal_{(t-k,j)}) = \delta_{ij}.
\end{equation}
\end{proposition}

\begin{proof}
When $i = j$, the result is immediate, so we now assume $i < j$.
Now we note $e_{j-t}(-\bal_{(t-k,j)}) \neq 0$ only if $k \geqslant 1$ by counting the number of variables, and so the claim trivially holds for all $k < 1$.
Next, we write the sum as $\sum_t (- 1)^{j - t} F_t$, where
\[
F_t := h_{t - i} (\alpha_{t - k}, \dotsc, \alpha_i)  \hspace{0.17em}
 e_{j - t} (\alpha_{t - k + 1}, \dotsc, \alpha_{j - 1}) = \sum_{(C, D) \in \Omega_t} \bal^C \bal^D,
\]
where $\Omega_t$ is the set of pairs $(C, D)$ with $C = \{c_1, \ldots, c_{t- i} \}$ being a multiset and $D = \{d_1, \ldots, d_{j - t} \}$ being a set of integers satisfying
\begin{equation}
	\label{eq:pqineq}
t - k \leqslant c_1 \leqslant \cdots \leqslant c_{t - i} \leqslant i,
 \qquad t - k < d_1 < \cdots < d_{j - t} < j.
\end{equation}
Additionally, we are using the notation $\bal^C$ to denote $\prod_m\alpha_{c_m}$.
Let $\Omega = \bigcup_t \Omega_t$.
We will define an involution $\eta$ of $\Omega$ that maps every element of $\Omega_t$ into either $\Omega_{t - 1}$ or $\Omega_{t + 1}$, and such that if $(C', D') = \eta (C, D)$, then $\bal^{C'} \bal^{D'} = \bal^C \bal^D$.
This is sufficient, since the contributions of terms corresponding by the involution will cancel in pairs, proving~\eqref{eq:heidentity}.

The involution $\eta$ will either move an element of $C$ to $D$, giving an element of $\Omega_{t-1}$ or it will move an element of $D$ to $C$, giving an element of $\Omega_{t+1}$.
The recipe is as follows:
\begin{itemize}
\item If $C$ and $D$ have distinct smallest elements, move this smallest element in $C \cup D$ from $C$ to $D$ or from $D$ to $C$.

\item If $C$ and $D$ have the same smallest element, move it from $D$ to $C$.
\end{itemize}
Note that $D$ is a set (without multiplicities), while $C$ is allowed to have multiplicities.
Moreover $C$ is allowed to have an element as small as $t - k$, while the smallest element of $D$ must be strictly larger than $t - k$.
With this in mind, it is easy to see that $\eta$ described by this recipe takes every element of $\Omega_t$ into either $\Omega_{t-1}$ or $\Omega_{t+1}$, and that $\eta$ has order~$2$.
It must be argued that if we move the smallest element $d_1$ of $D$ into $C$, then $d_1\leqslant i$. If $C$ is nonempty, then $d_1\leqslant c_1\leqslant i$, as required.
On the other hand if $C$ is empty, then $t=i$, and (\ref{eq:pqineq}) implies that $d_1 \leqslant t$, so $d_1\leqslant i$ in this case also.
\end{proof}

We also have another distinct identity that is a simple consequence of the change of bases from Proposition~\ref{prop:shifted_cob}.

\begin{corollary}
\label{cor:product_cob_identity}
For fixed $m \geqslant j \geqslant 0$, we have
\begin{align*}
\sum_{k=0}^j h_k(\bal_{[1,m-k+1]}) e_{j-k}(-\bal_{(0,m-k+1)}) & = \delta_{j0},
&
\sum_{k=0}^j e_k(-\bal_{(0,m+1)}) h_{j-k}(\bal_{[1,m-j+1]}) & = \delta_{j0}.
\end{align*}
\end{corollary}

\begin{proof}
The first follows from expanding $z^m$ in terms of the shifted powers and then back again into the usual powers.
The second is the shifted powers to normal powers and back to shifted.
In more detail, the first identity is
\begin{align*}
z^{-m} & = \sum_{k=0}^m h_{m-k}(\bal_{[1,k+1]}) (z^{-1}|\bal)^k
 = \sum_{k=0}^m h_{m-k}(\bal_{[1,k+1]}) \sum_{j=0}^k e_{k-j}(-\bal_{(0,k+1)}) z^{-j}
\\ & = \sum_{j=0}^m \sum_{k=j}^m h_{m-k}(\bal_{[1,k+1]}) e_{k-j}(-\bal_{(0,k+1)}) z^{-j}
\\ & = \sum_{j=0}^m \sum_{k=0}^{m-j} h_k(\bal_{[1,m-k+1]}) e_{m-j-k}(-\bal_{(0,m-k+1)}) z^{-j}  
\\ & = \sum_{j=0}^m \sum_{k=0}^j h_k(\bal_{[1,m-k+1]}) e_{j-k}(-\bal_{(0,m-k+1)}) z^{m-j}.  
\end{align*}
The second identity is similar.
\end{proof}

\section{Lie algebra representations}
\label{sec:lie_algebra}

The \defn{fermionic Fock space} $\fermionfock$ is the semi-infinite wedge product of vectors in the free module $V = \bigoplus_{i \in \ZZ} \CC[\bal] v_i$ that satisfy
\[
v_{i_1} \wedge v_{i_2} \wedge \cdots,
\qquad \qquad
\text{ where } i_k = m + k \text{ for all } k \gg 1 \text{ and some } m \in \ZZ.
\]
The $m \in \ZZ$ is called the \defn{charge} and defines a grading $\fermionfock = \bigoplus_{m \in \ZZ} \fermionfock^m$.
A basis of $\fermionfock^m$ is indexed by all partitions, denoted $\mcP$, with the basis vector indexed by $\lambda \in \mcP$ defined by
\[
\ket{\lambda}_m = v_{\lambda_1} \wedge v_{\lambda_2-1} \wedge \cdots.
\]
When $m = 0$, we denote $\ket{\lambda} := \ket{\lambda}_0$ for brevity.
Let ${}_{m} \bra{\lambda}$ denote the dual vector to $\ket{\lambda}_m$, and the natural pairing ${}_{\ell} \braket{\mu}{\lambda}_m = \delta_{\ell m} \delta_{\mu\lambda}$ is such that ${}_{\ell} \bra{\mu} X \ket{\lambda}_m$ is unambiguous for any operator $X$ on $\fermionfock$.

Let $\gl_{\infty}$ be the Lie algebra on the free $\CC[\bal]$-module $\bigoplus_{i,j \in \ZZ} \CC[\bal] E_{ij}$ with commutation law
\[
[E_{ij}, E_{k\ell}] = \delta_{jk} E_{i\ell} - \delta_{i\ell} E_{kj}.
\]
We may think of these as endomorphisms of $V$, where $E_{ij}$ is the endomorphism that maps $v_j$ to $v_i$ and annihilates all other basis vectors.
Although $\gl_{\infty}$ is a Lie algebra, we may also regard it as an associative ring (without unit) with multiplication $E_{ij} E_{k\ell} = \delta_{jk} E_{i\ell}$ by realizing it as a (nonunital) subalgebra of $\End(V)$.
Furthermore, $\gl_{\infty}$ has a natural action on $\fermionfock$ by
\[
M \ket{U} = M\bigl (u_1 \wedge u_2 \wedge \cdots \bigr) = \bigl( (M u_1) \wedge u_2 \wedge \cdots \bigr) + \bigl(u_1 \wedge (M u_2) \wedge \cdots \bigr) + \cdots
\]
for any $\ket{U} \in \fermionfock$ coming from $M \in \gl_\infty$, which act on basis vectors of $V$ as usual with $E_{ij}v_k = \delta_{jk} v_i$.
In other words, we have a representation $r \colon \gl_{\infty} \to \End(\fermionfock)$.
However, $\gl_{\infty}$ is clearly not all of $\End(V)$ since $\gl_{\infty}$ does not contain the identity map.

\subsection{Near diagonal infinite matrices}

Before we get to the space we consider in this paper, we describe an enlarged Lie algebra that has appeared in, \textit{e.g.},~\cite[Ch.~14]{KacInfinite} to described the classical boson-fermion correspondence (see also \cite[Ch,~4--6]{KacRaina}).
Let $\overline{\mfa}_{\infty}$ be the space of formal sums $A = (a_{i j}) := \sum_{ij} a_{i j} E_{i j}$, where $a_{i j} = 0$ unless $\abs{i - j} < N$ for some $N$ depending on~$A$.
This space is naturally a ring (with unit) by extending the product for $\gl_{\infty}$, which allows us to naturally identity $\overline{\mfa}_{\infty} = \End(V)$.
As such, we can consider $\overline{\mfa}_{\infty}$ as the space of matrices with finitely many nonzero diagonals.

\subsection{Current operators and near triangular infinite matrices}

Next, we define the \defn{deformed current operators}
\[
J_k^{(\bal)} := \sum_{i,j}A_{ij}^k E_{ij},
\quad
\text{ where }
A_{ij}^k = \begin{cases}
	e_{j-i-k}(-\bal_{(i,j)}) & \text{if } j \geqslant i + k \text{ and } k > 0, \\
	h_{j-i-k}(\bal_{[j,i]}) & \text{if } j \leqslant i \leqslant j - k \text{ and } k \leqslant 0, \\
	0 & \text{otherwise.}
	\end{cases}
\]
In particular, we have $A^0_{ij} = \delta_{ij}$ and for $k > 0$, we have $A_{ij}^k = 0$ (resp.\ $A_{ij}^{-k} = 0$) whenever $j < i$ (resp.\ $j > i$).
Additionally, for all $i,j,k \in \ZZ$ we can write
\begin{equation}
\label{eq:a_int}
A_{ij}^k 
= \oint z^{k-1} (z^{-1}|\sigma^i\bal)^{j-i-1} \frac{dz}{2\pi\ii z},
\end{equation}
from~\cite[Lemma~4.5]{BHS} (note that we need to take $z \mapsto z^{-1}$ to get the contours to match; \textit{cf.}\ Remark~\ref{rem:formal_contours}).

However, $J_k^{(\bal)} \notin \overline{\mfa}_{\infty}$, and so we need an even larger ring consisting of elements close to upper triangular matrices.
Formally, let $\overline{\mfa}_{\infty}^+$ be the larger ring of $A = \sum a_{ij} E_{ij}$, where $a_{ij} = 0$ unless $i - j < N$ for some $N$ (again depending on $A$).
Elements of $\overline{\mfa}_{\infty}^+$ are no longer necessarily endomorphisms of $V$ if we specialize $\bal$ (say, to complex numbers) but could be considered as linear maps $V \to \widehat{V} := \prod_{i \in \ZZ} \CC[\bal] v_i$.
Nevertheless, for any $A, B \in \overline{\mfa}_{\infty}^+$ the sum $\sum_j a_{ij} b_{jk}$ is finite.
Thus $\overline{\mfa}_{\infty}^+$ is an associative $\CC[\bal]$-algebra (with unit) and therefore a Lie algebra (over $\CC[\bal]$) with Lie bracket $[A, B] = A B - B A$.

Our first goal is to show that the deformed current operators commute as elements in $\overline{\mfa}_{\infty}^+$.
We will see later that we will recover the Heisenberg relations by taking a central extension (which essentially is accounting for the effects of the normal ordering; see Remark~\ref{rem:normal_ordering} below).

\begin{proposition}[{\cite[Prop.~6.6]{BHS}}]
\label{prop:commuting_currents}
As elements in the associative algebra $\overline{\mfa}_{\infty}^+$, for all $k, \ell \in \ZZ$,
\[
J_k^{(\bal)} J_{\ell}^{(\bal)} = J_{k+\ell}^{(\bal)}.
\]
Moreover, $[J_k^{(\bal)}, J_{\ell}^{(\bal)}] = 0$ and the inverse of $J_k^{(\bal)}$ is $J_{-k}^{(\bal)}$.
\end{proposition}

Like~\cite[Rem.~6.7]{BHS}, while one could say $J_k^{(\bal)} \cdot J_{\ell}^{(\bal)} = J_{k+\ell}^{(\bal)}$ \textit{as matrices}, this is misleading given that we want to consider representations of $\mfa_{\infty}^+$ as a Lie algebra.
As for the proof of~\cite[Prop.~6.6]{BHS}, the proof of Proposition~\ref{prop:commuting_currents} is equivalent to showing, for all $k,\ell, p,q \in \ZZ$,
\begin{equation}
\label{eq:aijkmatrix}
\sum_{a \in \ZZ} A_{p, a}^k A_{a,q}^{\ell} = A_{p,q}^{k+\ell},
\end{equation}
which was~\cite[Prop.~6.8]{BHS}.
Unlike~\cite[Prop.~6.8]{BHS}, we will not use any analytic assumptions nor will we use formal distribution calculus (when $\bbe = 0$) to show~\eqref{eq:aijkmatrix}.
Hence, we have an entirely new proof of~\eqref{eq:aijkmatrix} (and thus Proposition~\ref{prop:commuting_currents}).

If $U$ is a subset of $\ZZ$, let us define a metric on $\ZZ - U$ by
\[
d_U (a, b) = \begin{cases}
     \abs{a - b} - \abs{(a, b) \cap U} & \text{if $b \geqslant a$},\\
     \abs{a - b} - \abs{(b, a) \cap U} & \text{if $b \leqslant a$,}
   \end{cases}
\]
where if $b \geqslant a$, then $(a, b)$ denotes the set of integers $j$ in the open interval $a < j < b$ (hence $(a, a) = \varnothing$).
We will also use $[a, b]$ to denote the set of integers in the closed interval $a\leqslant j\leqslant b$.

\begin{lemma}
\label{lemma:aijkmatrix_pos}
Equation~\eqref{eq:aijkmatrix} holds for all $k, \ell > 0$.
\end{lemma}

\begin{proof}
Assume $j \geqslant i + k$. Then
\begin{equation}
\label{eq:aijkchar}
A_{ij}^k = (-1)^{i-j-k}
\sum_{i < s_1 < \cdots < s_{j - i - k} < j} 
\alpha_{s_1} \cdots \alpha_{s_{j - i - k}} .
\end{equation}
For each term in the sum~\eqref{eq:aijkchar}, if $S = \{s_1, \dotsc, s_{j - i - k} \}$, there are exactly $k - 1$ elements of $(i, j) - S$.
So $k = d_S(i, j)$.

Using this characterization, we prove~\eqref{eq:aijkmatrix} for $k, \ell > 0$.
Both sides vanish unless $q - p \geqslant k + \ell$, and we assume this.
By~\eqref{eq:aijkchar}, we need to show
\begin{equation}
\label{eq:pqterm}
\sum_{p < u_1 < \cdots < u_{q - p - k - \ell} < q} \alpha_{u_1} \cdots \alpha_{u_{q - p - k - \ell}}
\end{equation}
equals
\[
\sum_j \left( \sum_{p < s_1 < \cdots < s_{j - p - k} < j} \alpha_{s_1} \cdots \alpha_{s_{j - p - k}} \right)
\left( \sum_{j < t_1 < \cdots < t_{q - j - \ell} < q} \alpha_{t_1} \cdots \alpha_{t_{q - j - \ell}} \right).
\]
Given $p < u_1 < \cdots < u_{q - p - k - \ell} < q$, we will show that there is a unique $j \in (p, q)$ such that the sequence $u_1, \dotsc, u_{q - p - k - \ell}$ splits up into two sequences, $s_1, \dotsc, s_{j - p - k}$ followed by $t_1, \dotsc, t_{q - j - \ell}$, where
\begin{equation}
\label{eq:divcon}
p < s_1 < \cdots < s_{j - p - k} < j,
\qquad\qquad
j < t_1 < \cdots < t_{q - j - \ell} < q.
\end{equation}
Thus we want $s_i = u_i$ and $t_i = u_{i + j - p - k}$, and the issue is to show that there is a unique choice of $j$ such that this is possible.
Let $U = \{u_{1,} \cdots, u_{q - p - k - \ell} \}$.
Then $d_U (p, q) = k + \ell$, so there is a unique $j$ in $(p, q) - U$ such that $d_U (p, j) = k$ and $d_U(j, q) = \ell$.
Clearly this is the unique $j$ that realizes the term~\eqref{eq:pqterm} as a product of two terms in $A_{p, j}^k A_{j, q}^{\ell}$.
This proves our claim.
\end{proof}

\begin{lemma}
\label{lemma:jkinverse}
Suppose that $k > 0$.
Then the matrices $J_k^{(\bal)}$ and $J_{- k}^{(\bal)}$ are inverses.
\end{lemma}

\begin{proof}
It is not hard to argue that the triangular matrix $J^{(\bal)}_k$ is invertible, so we only need to compute one or the other of the identities
$J_{- k}^{(\bal)} J_k^{(\bal)} = I$ and $J_k^{(\bal)} J_{-k}^{(\bal)} = I$.
Thus we want to prove
\begin{equation}
\label{eq:akinvid}
\sum_t A_{i, t}^{- k} A_{t, j}^k = \delta_{ij} .
\end{equation}
This is equivalent to~\eqref{eq:heidentity} after replacing $t$ by $t - k$.
\end{proof}

Lemma~\ref{lemma:jkinverse} is of limited significance since we will not be considering the matrix structure of $J_k^{(\bal)}$.
However it does allow us to finish our proof of Proposition~\ref{prop:commuting_currents}.

\begin{proof}[Proof of Proposition~\ref{prop:commuting_currents}]
Since $J_0 = I$ is central, we may asssume $k$ and $\ell$ are nonzero.
If they are both positive, then the claim holds by Lemma~\ref{lemma:aijkmatrix_pos}.
Now multiplying the identity $[J_k^{(\bal)}, J_{\ell}^{(\bal)}] = 0$ left and right by $J_{-\ell}^{(\bal)}$ and using Lemma~\ref{lemma:jkinverse}, we obtain $[J_k^{(\bal)}, J_{-\ell}^{(\bal)}] = 0$.
Repeating this procedure gives $[J_{-k}^{(\bal)}, J_{-\ell}^{(\bal)}] = 0$.
\end{proof}

Analogous to~\cite[Cor.~6.9]{BHS}, we obtain the generalizations of classical plethysm formulas and Newton identities from Proposition~\ref{prop:commuting_currents}.

\begin{corollary}[{\cite[Cor.~6.9]{BHS}}]
For $k>0$ and $d,a,b\in\ZZ$, then
\begin{subequations}
\begin{align}
\sum_{r+t = d} e_r(-\bal_{(a,a+r+k)}) e_t(-\bal_{(a+r+k,b)}) & = e_d(-\bal_{(a,b)}),
\allowdisplaybreaks \\
\sum_{s+u = d} h_u(\bal_{[b,b+u+k]}) h_s(\bal_{[b+u+k,a]}) & = h_d(\bal_{[b,a]}),
\allowdisplaybreaks \\
\sum_{r-u = d} e_r(-\bal_{(a,r+a+k)}) h_u(\bal_{[b,r+a+k]}) & = \begin{cases}
  e_d(-\bal_{(a,b)}) & \text{if } d < b-a, \\
  h_{-d}(\bal_{[b,a]}) & \text{if } d > b-a, \\
  \delta_{ab} & \text{if } d=b-a.
\end{cases}
\end{align}
\end{subequations}
\end{corollary}

\subsection{Central extensions}

Next, we modify $\overline{\mfa}_{\infty}^+$ by adding a cocycle, which we use to introduce a central extension following~\cite{KacRaina}.
In this setting, we will see that the deformed current operators almost commute; more specifically, they will satisfy the Heisenberg relations.

There is the cocycle $\varphi$ on $\gl_{\infty}$ defined by
\begin{equation}
\label{eq:cocycle}
\varphi (E_{ij}, E_{ji}) = - \varphi (E_{ji}, E_{ij}) =
\begin{cases}
 1 & \text{if $i \leqslant 0, j > 0$,}\\
 -1 & \text{if $i>0$, $j\leqslant 0$}\\
 0 & \text{in all other cases,}
\end{cases}
\end{equation}
We note that for $A = (a_{i j})$ and $B = (b_{i j})$ in $\overline{\mfa}^+_{\infty}$, we (formally) expand $\varphi(A,B)$ linearly as
\[
\sum_{i,j,k,l} a_{ij}b_{kl}\varphi(E_{ij},E_{kl}) = \sum_{i,j} a_{ij}b_{ji}\varphi(E_{ij},E_{ji}) = \sum_{\substack{i \leqslant 0 \\ j > 0}} a_{ij}b_{ji} - \sum_{\substack{i > 0\\ j \leqslant 0}} a_{ij}b_{ji}.
\]
Since $a_{ij}$ vanishes unless $i-j < N_A$ for some $N_A$, and $b_{ji}$ vanishes unless $j-i < N_B$ for some $N_B$, the support of the sum is restricted to a band $\abs{i-j} \leqslant \max(N_A, N_B)$.
This support intersects the two quadrants in a finite number of terms.
Hence the sum is finite, so $\varphi$ can be extended by linearity to~$\overline{\mfa}^+_{\infty}$.

The cocycle $\varphi$ is indeed an element of the group $Z^2(\overline{\mfa}^+_{\infty}, \CC^{\times})$, meaning that it is skew-symmetric and satisfies the cocycle relation:
\[
\varphi([A, B], C) + \varphi([B, C], A) + \varphi([C, A], B) = 0.
\]
This is easily checked if $A, B, C \in \gl_{\infty}$ as $\varphi$ is a coboundary for $\gl_{\infty}$:
\[
 \varphi(A, B) = f([A, B]),
 \qquad\qquad \text{ where }
 f (E_{i j}) = \begin{cases}
 1 & \text{if $i = j \leqslant 0$,}\\
 0 & \text{otherwise},
\end{cases}
\]
and extended by linearity.
Thus for $A, B, C \in \gl_{\infty}$, the cocycle relation follows from the Jacobi identity.
The cocycle relation then follows by linearity for all $A, B, C \in \overline{\mfa}_{\infty}^+$, even though the function $f$ does not extend to $\overline{\mfa}_{\infty}^+$.

Using the cocycle $\varphi$, we construct a central extension $\mfa_{\infty}^+$ of $\overline{\mfa}_{\infty}^+$.
As a Lie algebra, the central extension is the vector space $\overline{\mfa}_{\infty} \oplus \CC \cdot \one$, where $\one$ is a central element, with the Lie bracket defined by
\[
[A, B] = AB - BA + \varphi(A, B) \cdot \one.
\]
All of these statements also hold for the ring $\overline{\mfa}_{\infty}^-$, the space of matrices $A = (a_{ij})$ such that $a_{ij} = 0$ unless $i - j > N$ for some $N$, which are matrices close to the lower triangular matrices.
We denote the corresponding central extension by $\mfa_{\infty}^-$.

As a historical remark, the corresponding central extension $\mfa_{\infty}$ of $\overline{\mfa}_{\infty}$ was introduced independently by Date, Jimbo, Kashiwara and Miwa~\cite{DateJimboKashiwaraMiwa} and by Kac and Peterson~\cite{KacPetersonSpin}.
The rings $\overline{\mfa}_{\infty}^{\pm}$ and their Lie algebras have made an appearance in~\cite[\S14.10]{KacInfinite}, where it was denoted by $\widetilde{\gl}_{\infty}$ and only considered as a Lie algebra.
However, as far as we are aware, the central extensions $\mfa_{\infty}^{\pm}$ have not appeared explicitly before in the literature.

The action of $\gl_{\infty}$ from the representation $r \colon \gl_{\infty} \to \End(\fermionfock)$ does not extend to $\overline{\mfa}_\infty$ because (for example) this would be divergent if $M = I$.
However, following~\cite[Ch.~4]{KacRaina}, we may modify it to obtain an action on $\overline{\mfa}_\infty$ or more generally on $\overline{\mfa}_\infty^+$.
First, we adjust the representation to obtain a projective representation $\widehat{r}\colon \gl_\infty \to \End(\fermionfock)$ by
\begin{equation}
\label{eq:rhat}
\widehat{r} (E_{i i}) = \begin{cases}
 r (E_{i i}) & \text{if $i > 0$},\\
 r (E_{i i}) - I & \text{if $i \leqslant 0$},
\end{cases}
\end{equation}
with $\widehat{r} (E_{i j}) = r (E_{i j})$ when $i \neq j$.
This modification eliminates the divergences, and $\widehat{r}$ extends by linearity to a projective representation of $\overline{\mfa}_{\infty}^{\pm}$.
We find that
\begin{equation}
\label{eq:central_extension}
\widehat{r} ([E_{i j}, E_{k l}]) = [\widehat{r} (E_{i j}), \widehat{r} (E_{kl})] + \varphi (E_{i j}, E_{k l}) \cdot \mathbf{1},
\end{equation}
where $\mathbf{1}$ acts as the identity operator.
Thus $\widehat{r} \colon \mfa_{\infty}^{\pm} \longrightarrow \End(\fermionfock)$ is a representation.

\begin{theorem}[{\cite[Thm.~4.4]{BHS}}]
\label{thm:heisenberg_relations}
As elements of $\mfa_{\infty}^+$, the deformed current operators $J_k^{(\bal)}$ generate a Heisenberg Lie algebra; that is, they satisfy
\[
\bigl[ J_k^{(\bal)}, J_l^{(\bal)} \bigr] = k \delta_{k, - l} \cdot \mathbf{1}.
\]
\end{theorem}

We will give an entirely different proof of Theorem~\ref{thm:heisenberg_relations} here by a direct computation in our algebra.
In order to show Theorem~\ref{thm:heisenberg_relations}, it is sufficient to show how the deformed current operators behave under the cocycle.
The result follows from~\eqref{eq:central_extension} and Proposition~\ref{prop:commuting_currents}.

\begin{proposition}
\label{prop:jkcocycle}
We have
\[
\varphi(J^{(\bal)}_k, J^{(\bal)}_{\ell}) = \begin{cases}
   k & \text{if $\ell = - k$,}\\
   0 & \text{otherwise}.
 \end{cases}
\]
\end{proposition}

\begin{proof}
By the definition of $\varphi$, we have
\begin{equation}
\label{eq:jkjlcom}
\varphi(J^{(\bal)}_k, J^{(\bal)}_l) = \sum_{i, j} A_{ij}^k A_{ji}^{\ell} \varphi(E_{ij}, E_{ji})
= \sum_{\substack{i \leqslant 0 \\ j > 0}} A_{ij}^k A_{ji}^{\ell} - \sum_{\substack{ i > 0\\ j \leqslant 0}} A_{ij}^k A_{ji}^{\ell}.
\end{equation}
It is not hard to see that both sums are finite due to the support conditions of $A_{ij}^k$, which we summarize.
If $k = 0$, then $A_{ij}^0 = \delta_{ij}$.
If $k > 0$ then $A_{ij}^k = 0$ unless $j \geqslant i + k$, and if $k < 0$, then $A_{ij}^k = 0$ unless $j \leqslant i \leqslant j - k$.
We now proceed in checking the claim case-by-case.

If either $k$ or $\ell$ is zero, $A_{ij}^k A_{ji}^{\ell} \neq 0$ implies $i = j$ and both terms in~\eqref{eq:jkjlcom} vanish.

If both $k, \ell > 0$, and $A_{ij}^k$ and $A_{ji}$ are nonzero, then $j \geqslant i + k$ and $i \geqslant j + \ell$.
Thus $i > j > i$, which is a contradiction, so again both terms in~\eqref{eq:jkjlcom} vanish.

If both $k, \ell < 0$, and if $A_{ij}^k$ and $A_{ji}^{\ell}$ are nonzero then $j \leqslant i \leqslant j - k$ and $i \leqslant j \leqslant i - \ell$, which is possible, but only if $i = j$, so again both terms in~\eqref{eq:jkjlcom} vanish.

We may thus assume that $k$ and $\ell$ have opposite signs, so without loss of generality assume $k > 0$ and $\ell < 0$.
Then $A_{ij}^k A_{ji}^{\ell} \neq 0$ implies that $j \geqslant i + k$ and $i \leqslant j \leqslant i - \ell$, which implies that $k \leqslant - \ell$.
The second term in~\eqref{eq:jkjlcom} vanishes because $i \leqslant j$ is incompatible with $i > 0$ and $j \leqslant 0$. In this case, substituting the definitions of $A_{ij}^k$ and $A_{ji}^{\ell}$ gives exactly the summation in Lemma~\ref{prop:cocyclecase}, (with $\ell$ replaced by its negative) and we are done.
\end{proof}

\section{Double supersymmetric functions}
\label{sec:double_funcs}

We describe the ring of double supersymmetric functions given by Molev~\cite{Molev09} but following the conventions in~\cite{BHS}.
In particular, to get the functions in~\cite{Molev09}, apply $\iota$ to the $\bal$ parameters in all of the formulas given here.
On the other hand, to match~\cite{MolevFactorialSupersymmetric}, we need to restrict to finitely many variables $\xx_n / \yy_n$ and then apply $\sigma^n$.

We take $\ZZ[\bal]$ as our coefficient ring unless otherwise stated.

We define the \defn{double powersum}, \defn{homogeneous}, and \defn{elementary double supersymmetric functions} as
\begin{align*}
p_k(\xx / \yy \dv \bal) & = p_k(\xx / \yy) = p_k(\xx) - p_k(-\yy) = \sum_{i=1}^{\infty} x_i^k - (-y_i)^k,
\\ h_k(\xx / \yy \dv \bal) & = \sum_{b+a=k} \sum_{\substack{i_1 \geqslant \cdots \geqslant i_b \\ j_1 < \cdots < j_a}} (y_{j_1} + \alpha_{1-j_1}) \cdots (y_{j_a} + \alpha_{a-j_a}) (x_{i_1} - \alpha_{a+1-i_1}) \cdots (x_{i_b} - \alpha_{k-i_b}),
\\ e_k(\xx / \yy \dv \bal) & = \sum_{b+a=k} \sum_{\substack{i_1 > \cdots > i_b \\ j_1 \leqslant \cdots \leqslant j_a}} (y_{j_1} + \alpha_{1-j_1}) \cdots (y_{j_a} + \alpha_{-j_a-a}) (x_{i_1} - \alpha_{-i_1-a-1}) \cdots (x_{i_b} - \alpha_{-i_b-k}).
\end{align*}
For $\bal = 0$, these become the classical supersymmetric functions, which further reduce to the classical symmetric functions when $\yy = 0$ (see, \textit{e.g.},~\cite{MacdonaldBook,ECII} for more details).
We remark that our double powersums are simply the supersymmetric powersum functions, but this is \emph{not} true for the double homogeneous and elementary cases.
Furthermore, the supersymmetric functions are a plethystic substitution $f[\xx - (-1)\yy]$ of the corresponding usual symmetric function $f(\xx) = f[\xx]$.
For a double supersymmetric function $f(\xx/\yy \dv \bal)$, we write the corresponding supersymmetric (resp.\ symmetric) function as $f(\xx/\yy)$ (resp.\ $f(\xx)$).

\begin{remark}
\label{rem:substitution}
For $\bal = \alpha$, the factorial supersymmetric functions are simply the supersymmetric functions with the substitution $x_i \mapsto x_i - \alpha$ and $y_j \mapsto y_j + \alpha$.
In other words, they reduce to the classical case when specializing $\bal$ to the same value.
\end{remark}

From~\cite[Thm.~2.1]{MolevFactorialSupersymmetric} (\textit{cf.}~\cite[Cor.~3.2]{Molev09} at $z_i = 0$ for $i > 1$ or~\cite[Eq.~(45)]{BHS} at $\bbe = 0$), we have the following generating series
\begin{subequations}
\label{eq:double_gf}
\begin{align}
\label{eq:double_h_gf}
\sum_{k=0}^{\infty} \frac{h_k(\xx / \yy \dv \bal)}{z^k (z^{-1}|\bal)^k} & = \prod_{i=1}^{\infty} \frac{1 + y_i z}{1 - x_i z}
= \exp \left( \sum_{k=1}^{\infty} \frac{1}{k} p_k(\xx/\yy\dv\bal)  z^k \right),
\\
\label{eq:double_e_gf}
\sum_{k=0}^{\infty} (-1)^k e_k(\xx / \yy \dv \bal) (z^{-1}|\bal)^{-k} & = \prod_{i=1}^{\infty} \frac{1 - x_i z}{1 + y_i z}
= \exp \left( -\sum_{k=1}^{\infty} \frac{1}{k} p_k(\xx/\yy\dv\bal) z^k \right),
\end{align}
\end{subequations}
where here we have followed~\cite{Molev09} but reindexed the $\bal$ parameters and taken the $m = n \to \infty$ (projective) limit (in particular, see~\cite[Eq.~(2.7)]{Molev09}).
Note that to obtain $n < \infty$ variables of $\xx$ and $\yy$, we simultaneously specialize to $\xx_n$ and $\yy_n$.
Then if we then want to remove an additional $\xx$ or $\yy$ variables, we either set $x_i = \alpha_i$ or $y_i = -\alpha_i$ as appropriate.
The right equalities in~\eqref{eq:double_gf} can be seen from plethystic substitutions $\xx \mapsto \xx - (-1)\yy$ from the classical formulas (see, \textit{e.g.}~\cite[Ch.~7]{ECII}).

\begin{example}
Using our definition of $h_k(\xx / \yy \dv \bal)$ with $i_1, j_a \leqslant 2$ and regrouping as in~\cite[Eq.~(2.12)]{Molev09}, we compute
\begin{align*}
h_2(\xx_1/\yy_1 \dv \bal) & = (x_1 - \alpha_0)(x_1 - \alpha_1) + (y_1 + \alpha_0)(x_1 - \alpha_1),
\\
h_2(\xx_2/\yy_2 \dv \bal) & = (x_2 - \alpha_{-1})(x_2 - \alpha_0) + (x_2 - \alpha_{-1})(x_1 - \alpha_1) + (x_1 - \alpha_0)(x_1 - \alpha_1)
\\ & \hspace{20pt} + \bigl((y_1 + \alpha_0) + (y_2 + \alpha_{-1}) \bigr) \bigl((x_2 - \alpha_0) + (x_1 - \alpha_1) \bigr) + (y_1 + \alpha_0)(y_2 + \alpha_0)
\end{align*}
One can verify that taking $x_2 = y_2 = 0$ in the second formula yields the first.
Recall from~\cite[Sec.~2.4]{Molev09} that an \defn{$\mathbb{A}$-tableau} is a weakly increasing tableau in the alphabet $\{1' < 1 < 2' < 2 < \cdots\}$ such that at most one primed (resp.\ unprimed) $i$ is in each row (resp.\ column) and after  applying $\iota$, we have
\begin{equation}
\label{eq:A_tableau_formula}
s_{\lambda/\mu}(\xx_n/\yy_n\dv\bal) = \sum_T \prod_{\substack{\bbb \in \lambda/\mu\\T(\bbb) \text{ unprimed}}} (x_{T(\bbb)} - \alpha_{c(\bbb)}) \prod_{\substack{\bbb \in \lambda/\mu\\T(\bbb) \text{ primed}}} (y_{T(\bbb)} + \alpha_{c(\bbb)}).
\end{equation}
Restricting to two variables, we compute
\begin{align*}
\ytableausetup{boxsize=1em}%
\ytableaushort{{1'}1} & \; (y_1+\alpha_0)(x_1-\alpha_1),
&
\ytableaushort{{1'}{2'}} & \; (y_1+\alpha_0)(y_2+\alpha_1),
&
\ytableaushort{{1'}2} & \; (y_1+\alpha_0)(x_2+\alpha_1),
\\
\ytableaushort{11} & \; (x_1-\alpha_0)(x_1-\alpha_1),
&
\ytableaushort{1{2'}} &\;  (x_1-\alpha_0)(y_2+\alpha_1),
&
\ytableaushort{12} & \; (x_1-\alpha_0)(x_2-\alpha_1),
\\
\ytableaushort{{2'}2} & \; (y_2+\alpha_0)(x_2-\alpha_1),
&
\ytableaushort{22} & \; (x_2-\alpha_0)(x_2-\alpha_1).
\end{align*}
It is a straightforward check that the sum of the weights of these tableaux is equal to $h_2(\xx_2 / \yy_2 \dv \bal)$.
\end{example}

We also have the classical (supersymmetric) formulas
\begin{align*}
\sum_{k=0}^{\infty} h_k(\xx / \yy) z^k & = \prod_{i=1}^{\infty} \frac{1 + y_i z}{1 - x_i z},
&
\sum_{k=0}^{\infty} (-1)^k e_k(\xx / \yy) z^k & = \prod_{i=1}^{\infty} \frac{1 - x_i z}{1 + y_i z}.
\end{align*}
Hence, we can then apply Proposition~\ref{prop:shifted_cob} with~\eqref{eq:double_gf} to express the double homogeneous (resp.\ elementary) supersymmetric functions in terms of the classical homogeneous (resp.\ elementary) supersymmetric functions (which can be obtained from setting $\bal = 0$).

\begin{proposition}
\label{prop:eh_expansions}
For $k > 0$, we have 
\begin{align*}
h_k(\xx / \yy \dv \bal) & = \sum_{m=1}^k e_{k-m}(-\bal_{(0,k)}) h_m(\xx / \yy),
&
e_k(\xx / \yy \dv \bal) & = \sum_{m=1}^k e_{k-m}(\bal_{(-k+1,1)}) e_m(\xx / \yy),
\\
h_k(\xx / \yy) & = \sum_{m=1}^k h_{k-m}(\bal_{[1,m]}) h_m(\xx / \yy \dv \bal),
&
e_k(\xx / \yy) & = \sum_{m=1}^k h_{k-m}(-\bal_{[1-m,0]}) e_m(\xx / \yy \dv \bal),
\end{align*}
\end{proposition}

\begin{proof}
We compute
\begin{align*}
\sum_{k=0}^{\infty} \frac{h_k(\xx / \yy \dv \bal)}{z^k(z^{-1}|\bal)^k} & = \prod_{i=1}^{\infty} \frac{1 + y_i z}{1 - x_i z} = \sum_{m=0}^{\infty} h_m(\xx / \yy) z^m
\\ & = 1 + \sum_{m=1}^{\infty} h_m(\xx / \yy) \sum_{k=m}^{\infty} \frac{e_{k-m}(-\alpha_1,\dotsc,-\alpha_{k-1})}{z^k(z^{-1}|\bal)^k},
\\ & = 1 + \sum_{k=1}^{\infty}  \frac{1}{z^k(z^{-1}|\bal)^k} \sum_{m=1}^k e_{k-m}(-\alpha_1,\dotsc,-\alpha_{k-1}) h_m(\xx / \yy),
\end{align*}
where we used Proposition~\ref{prop:shifted_cob} after applying $\iota\sigma$.
The claim for $h_k(\xx / \yy \dv \bal)$ follows from equating coefficients of $\frac{1}{z^k(z^{-1}|\bal)^k}$.
The other claims are proved similarly.
\end{proof}

\begin{example}
Clearly $h_0(\xx/\yy\dv\bal) = e_0(\xx/\yy\dv\bal) = h_0(\xx/\yy) = e_0(\xx/\yy) = 1$.
It is straightforward to see that $h_1(\xx/\yy \dv \bal) = h_1(\xx/\yy)$ and $e_1(\xx/\yy \dv \bal) = e_1(\xx/\yy)$.
Next, we have
\begin{align*}
h_2(\xx_1 / \yy_1 \dv \bal) & = x_1^2 + x_1 y_1 - \alpha_1 (x_1 + y_1) = h_2(\xx_1 / \yy_1) + e_1(-\alpha_1) h_1(\xx_1 / \yy_1),
\\ h_3(\xx_1 / \yy_1 \dv \bal) & = (x_1 - \alpha_0) (x_1 - \alpha_1) (x_1 - \alpha_2) + (y_1 + \alpha_0) (x_1 - \alpha_1) (x_1 - \alpha_2)
\\ & = x_1^3 + x_1^2 y_1 - (\alpha_1 + \alpha_2)(x_1^2 + x_1 y_1) + \alpha_1 \alpha_2 (x_1 + y_1)
\\ & = h_3(\xx_1 / \yy_1) + e_1(-\alpha_1, -\alpha_2) h_2(\xx_1 / \yy_1) + e_2(-\alpha_1, -\alpha_2) h_1(\xx_1 / \yy_1),
\\ e_2(\xx_1 / \yy_1 \dv \bal) & = (y_1 + \alpha_1)(y_1 + \alpha_2) + (y_1 + \alpha_1)(x_1 - \alpha_2)
\\ & = x_1 y_1 + y_1^2 + \alpha_0(x_1 + y_1) = e_2(\xx_1 / \yy_1) + e_1(\alpha_0) e_1(\xx_1 / \yy_1),
\\ e_3(\xx_1 / \yy_1 \dv \bal) & = (y_1 + \alpha_0)(y_1 + \alpha_{-1})(y_1 + \alpha_{-2}) + (y_1 + \alpha_0)(y_1 + \alpha_{-1})(x_1 - \alpha_{-2})
\\ & = x_1 y_1^2 + y_1^3 + (\alpha_{-1} + \alpha_0)(x_1 y_1 + y_1^2) + \alpha_{-1} \alpha_0 (x_1 + y_1)
\\ & = e_3(\xx_1 / \yy_1) + e_1(\alpha_{-1}, \alpha_0) e_2(\xx_1 / \yy_1) + e_2(\alpha_{-1}, \alpha_0) e_1(\xx_1 / \yy_1).
\end{align*}
\end{example}

For partitions $\lambda, \mu$ of length at most $\ell$, we define the \defn{double Schur functions} by the Jacobi--Trudi-type formula~\cite[Eq.~(2.10), (2.11)]{Molev09} (\textit{cf.}~\cite[Thm.~3.1,~3.3]{MolevFactorialSupersymmetric})
\begin{align*}
s_{\lambda/\mu}(\xx/\yy \dv \bal) & := \det \bigl[ h_{\lambda_i-\mu_j-i+j}(\xx/\yy \dv \sigma^{\mu_j-j+1} \bal) \bigr]_{i,j=1}^{\ell}
\\
& = \det \bigl[ e_{\lambda_i'-\mu_j'-i+j}(\xx/\yy \dv \sigma^{j-\mu_j'-1} \bal) \bigr]_{i,j=1}^{\ell}.
\end{align*}

\begin{example}
We have
\begin{align*}
s_{22/1}(\xx/\yy \dv \bal) & = 
\det \begin{bmatrix}
h_1(\xx/\yy \dv \sigma\bal) & h_3(\xx/\yy \dv \sigma^{-1}\bal) \\
h_0(\xx/\yy \dv \sigma\bal) & h_2(\xx/\yy \dv \sigma^{-1}\bal)
\end{bmatrix}
\allowdisplaybreaks
\\ & = h_1(\xx/\yy \dv \sigma^{-1}\bal) h_2(\xx/\yy \dv \sigma\bal) - h_3(\xx/\yy \dv \sigma^{-1}\bal)
\allowdisplaybreaks
\\ & = s_{21}(\xx/\yy\dv\bal) + (\alpha_1 - \alpha_0) s_2(\xx/\yy\dv\bal)
\\ & \hspace{20pt} + (\alpha_1 - \alpha_0) s_{11}(\xx/\yy\dv\bal) + (\alpha_1 - \alpha_0)^2 s_1(\xx/\yy\dv\bal)
\allowdisplaybreaks
\\ & = s_{21}(\xx/\yy) + \alpha_1 s_2(\xx/\yy) - \alpha_0 s_{11}(\xx/\yy) - \alpha_0 \alpha_1 s_1(\xx/\yy).
\end{align*}
Next, we restrict to one pair of variables, where the above results in
\begin{align*}
\ytableausetup{boxsize=1em}%
s_{22/1}(\xx_1/\yy_1 \dv \bal) & =
\bigl( (x_1 - \alpha_1) + (y_1 + \alpha_1) \bigr)
\bigl( (x_1 - \alpha_{-1})(x_1 - \alpha_0) + (y_1 + \alpha_{-1})(x_1 - \alpha_0) \bigr)
\\ & \hspace{20pt} - \bigl( (x_1 - \alpha_{-1})(x_1 - \alpha_0)(x_1-\alpha_1) + (y_1 + \alpha_{-1})(x_1 - \alpha_0)(x_1 - \alpha_1\bigr)
\\ & = (x + y)(x - \alpha_0)(y + \alpha_1).
\end{align*}
On the other hand, the $\mathbb{A}$-tableau formula~\eqref{eq:A_tableau_formula} yields
\[
\ytableaushort{{\none}{1'},{1'}1} \quad (y_1 + \alpha_{-1})(x_1 - \alpha_0)(y_1 + \alpha_1),
\qquad
\ytableaushort{{\none}{1'},11} \quad (x_1 - \alpha_{-1})(x_1 - \alpha_0)(y_1 + \alpha_1),
\]
where it is easy to see the functions are equal.
\end{example}

When we specialize to a single set of variables by $\yy = -\bal$ (\textit{i.e.}, $y_i = -\alpha_i$) and then take a finite set of variables $\overline{\xx}_n = (x_1, \dotsc, x_n, \alpha_{n+1}, \cdots)$, we have a bialternant formula for the double Schur functions~\cite{Macdonald92,MiyauraMukaihiraGeneralized,Molev09} 
\[
s_{\lambda}(\overline{\xx}_n \dv \bal) = \frac{A_{\lambda+\delta}}{A_{\delta}},
\qquad\qquad
A_{\mu} = \det \bigl[ (x_i | \sigma^{n-j} \bal)^{\mu_j} \bigr]_{i,j=1}^{\ell} = \det \bigl[ (x_i | \sigma^n \sigma \bal)^{\mu_j} \bigr]_{i,j=1}^{\ell},
\]
where $\delta = (\ell-1, \dotsc, 2, 1, 0)$ is the staircase partition.

Recall the involution $\omega$ that sends $h_k \mapsto e_k$ for all $k \in \ZZ_{\geqslant 0}$.
By Proposition~\ref{prop:eh_expansions}, we can immediately recover~\cite[Cor.~5.14]{BHS} when $\bbe = 0$, and note that this is a completely new proof.

\begin{proposition}
\label{prop:omega}
We have
\[
\omega h_k(\xx/\yy\dv\bal) = e_k(\xx/\yy\dv{-\iota \bal}).
\]
\end{proposition}

Note that $\iota_- \omega$, where $\iota_-$ applies $-\iota$ to the $\bal$ parameters, yields the involution from Molev~\cite{MolevFactorialSupersymmetric,Molev09}.

Next, we use Proposition~\ref{prop:shifted_action} to describe how the $\bal$ shift operator $\sigma$ acts on the double elementary and homogeneous symmetric functions.
Note that the right hand sides of both generating series~\eqref{eq:double_gf} do \emph{not} depend on $\bal$.
Thus, we can compute the following.

\begin{proposition}[{\cite[Eq.~(8)]{Fun12},\cite[Lemma~2.4]{Molev09II}}]
\label{prop:shift_eh}
We have
\begin{subequations}
\begin{align}
e_k(\xx/\yy\dv\sigma\bal) & = e_k(\xx/\yy\dv\bal) + (\alpha_1 - \alpha_{2-k}) e_{k-1}(\xx/\yy\dv\bal),
\\
h_k(\xx/\yy\dv\sigma^{-1}\bal) & = h_k(\xx/\yy\dv\bal) + (\alpha_{k-1} - \alpha_0) h_{k-1}(\xx/\yy\dv\bal).
\end{align}
\end{subequations}
\end{proposition}

\begin{proof}
To show the formula for $h_k(\xx/\yy\dv\sigma\bal)$, we first note that
\begin{align*}
e^{\xi(\xx/\yy;z)} & = \sum_{k=0}^{\infty} \frac{h_k(\xx/\yy\dv\sigma^{-1}\bal)}{(z^{-1}|\sigma^{-1}\bal)^k}
= \sum_{k=0}^{\infty} \frac{h_k(\xx/\yy\dv\bal)}{(z^{-1}|\bal)^k}
\\ & = \sum_{k=0}^{\infty} {h_k(\xx/\yy\dv\bal)} \left( \frac{1}{(z^{-1}|\sigma^{-1}\bal)^k} + \frac{\alpha_k - \alpha_0}{(z^{-1}|\sigma^{-1}\bal)^{k+1}} \right)
\end{align*}
by~\eqref{eq:double_h_gf}. 
The claim follows from equating the coefficients of $\frac{1}{(z^{-1}|\sigma^{-1}\bal)^k}$.
The proof for $e_k(\xx/\yy\dv\sigma\bal)$ is similar, yielding
\[
\sum_{k=0}^{\infty} (-1)^k e_k(\xx/\yy\dv\bal) (w^{-1}|\sigma\bal)^{-k}
= \sum_{k=0}^{\infty} (-1)^k e_k(\xx/\yy\dv\bal) \left( (w^{-1}|\sigma\bal)^{-k} + (\alpha_{1-k} - \alpha_1) (w^{-1}|\sigma\bal)^{-k-1} \right)
\]
and then equating the coefficients of $(w^{-1}|\sigma\bal)^{-k}$ to obtain the result.
\end{proof}

Note that we can compute $e_k(\xx/\yy\dv\sigma\bal)$ and $h_k(\xx/\yy\dv\sigma^{-1}\bal)$ from Proposition~\ref{prop:shift_eh}, yielding the recursion relations and summation formulas
\begin{align*}
e_k(\xx/\yy\dv\sigma^{-1}\bal) & = e_k(\xx/\yy\dv\bal) - (\alpha_{k+2} - \alpha_1) e_{k-1}(\xx/\yy\dv\sigma\bal) = \sum_{m=0}^{k-1} \frac{e_{k-m}(\xx/\yy\dv\bal)}{(\alpha_0|\sigma^{m-k}\bal)^m},
\\
h_k(\xx/\yy\dv\sigma\bal) & = h_k(\xx/\yy\dv\bal) - (\alpha_k - \alpha_1) h_{k-1}(\xx/\yy\dv\sigma\bal) = \sum_{m=0}^{k-1} (\alpha_1|\sigma^{k-m}\bal)^m h_{k-m}(\xx/\yy\dv\bal).
\end{align*}
Note that these are finite sums since
\[
h_0(\xx/\yy\dv\sigma\bal) = h_0(\xx/\yy\dv\bal) = e_0(\xx/\yy\dv\sigma^{-1}\bal) = e_0(\xx/\yy\dv\bal) = 1.
\]

\section{The Boson-Fermion Correspondence} 
\label{sec:fermion_fields}

The \defn{Clifford algebra} $\mcC$ is generated by operators $\psi_i$ for $i \in \ZZ$ (resp.\ $\psi_j^*$ for $j \in \ZZ$) acting on $\fermionfock$ that adds (resp.\ deletes) a particle at site $i$ (resp. $j$) and acts by $0$ if it cannot.
These satisfy the canonical anticommutation relations
\[
[\psi_i, \psi_j]_+ = [\psi_i^*, \psi_j^*]_+ = 0,
\qquad\qquad
[\psi_i, \psi_j^*]_+ = \delta_{ij},
\]
where $[a, b]_+ = ab + ba$.
We define the \defn{deformed fermion fields} by
\[
\psi(z|\bal) := \sum_{i \in \ZZ} \frac{1}{(z^{-1}|\bal)^i} \psi_i = \sum_{i \in \ZZ} (z^{-1}|\sigma^i \bal)^{-i} \psi_i,
\qquad\quad
\psi^*(w|\bal) := \sum_{j \in \ZZ} w^{-1}(w^{-1}|\bal)^{j-1} \psi_j^*.
\]
There is also a normal ordering on the Clifford algebra defined by
\begin{subequations}
\label{eq:normal_ordering}
\begin{align}
\normord{\psi(z|\bal) \psi^*(w|\bal)} & = \psi(z|\bal) \psi^*(w|\bal) - \bra{\varnothing} \psi(z|\bal) \psi^*(w|\bal) \ket{\varnothing},
\\
\normord{\psi^*(w|\bal) \psi(z|\bal)} & = \psi^*(w|\bal) \psi(z|\bal) - \bra{\varnothing} \psi^*(w|\bal) \psi(z|\bal) \ket{\varnothing}.
\end{align}
\end{subequations}
Next, we note that $\psi(z|\bal)$ and $\psi^*(w|\bal)$ makes sense when expanding each term as a formal Laurent series in $z$, accepting formal infinite sums of basis elements in $\mcC$.

Now we provide an algebraic proof of the key result~\cite[Prop.~4.1]{BHS}, which was used in all of the subsequent proofs.

\begin{theorem}
\label{thm:main_theorem}
We have
\begin{align*}
\bra{\varnothing} \psi(z|\bal) \psi^*(w|\bal) \ket{\varnothing} & = \frac{z}{z-w},
&
\bra{\varnothing} \psi^*(w|\bal) \psi(z|\bal) \ket{\varnothing} & = \frac{z}{w-z},
\end{align*}
where this is an equality of functions $\CC^2 \to \CC$ if $\abs{z} > \abs{w}$ or $\abs{z} < \abs{w}$, respectively.
\end{theorem}

\begin{proof}
We now prove the first formula.
We can consider the expansion in the ring $\mcC[\bal][z^{-1}]\FPS{w}$ as $\psi_i \ket{\varnothing} = 0$ for all $i > 0$.
Using the fact that $\bra{\varnothing} \psi_i \psi_j^* \ket{\varnothing} = \delta_{ij}$, we have
\begin{align*}
\bra{\varnothing} \psi(z|\bal) \psi^*(w|\bal) \ket{\varnothing} & = w^{-1} \sum_{a=0}^{\infty} \frac{(w^{-1}|\bal)^{-a-1}}{(z^{-1}|\bal)^{-a}}
= \sum_{a=0}^{\infty} \frac{w^a}{z^a} (1 - \alpha_a w)^{-1} \prod_{i=0}^{a-1} \frac{1 - \alpha_i z}{1 - \alpha_i w}
\\ & = \sum_{a=0}^{\infty} \frac{w^a}{z^a} \sum_{k=0}^a e_k(-\bal_{(-1,a)}) z^k \sum_{m=0}^{\infty} h_m(\bal_{[0,a]}) w^m
\\ & = \sum_{b=0}^{\infty} w^b \sum_{q=0}^b z^{-q} \sum_{r=q}^b h_{b-r}(\bal_{[0,r]}) e_{r-q}(-\bal_{(-1,r)}),
\end{align*}
where in the last step, $b := a+m,\; q := a-k$, and $r := a$.
Finally, apply Corollary~\ref{cor:product_cob_identity} (replacing $r \mapsto b - r$ and taking $j = b - r$ to get it in the same form) to obtain
\[
\bra{\varnothing} \psi(z|\bal) \psi^*(w|\bal) \ket{\varnothing} = \sum_{b=0}^{\infty} \left( \frac{w}{z} \right)^b = \frac{1}{1 - w/z} = \frac{z}{z - w}.
\]

The proof for the second formula is similar.
\end{proof}

\begin{remark}
\label{rem:normal_ordering}
As noted just before~\cite[Prop.~6.6]{BHS}, the normal ordering~\eqref{eq:normal_ordering} roughly corresponds to taking the central extension $\mfa^{\pm}$ of $\overline{\mfa}^{\pm}$.
In more detail, consider the identification $E_{ij} \leftrightarrow \psi_i \psi_j^*$.
The normal ordering can be considered analogous to moving from the representation $r$ to $\widehat{r}$ as $\normord{E_{ij}} = E_{ij}$ unless $i = j \leqslant 0$, in which case $\normord{E_{ii}} = -\psi_i^* \psi_i = \psi_i \psi_i^* - 1 = E_{ii} - 1$.
This is not precise, which can be seen by comparing $\normord{[E_{ij}, E_{ab}]} = 0$ (as everything under the normal ordering (skew)commutes) and $[\normord{E_{ij}}, \normord{E_{ab}}] = \delta_{ja} E_{ib} - \delta_{ib} E_{aj}$.
\end{remark}

There is a deformed version of the Clifford algebra shift operator $\Sigma_{(\bal)} := \Sigma_{(\bal,0)}$ from~\cite[Sec.~4.2]{BHS}, which is given by the adjoint action essentially acting as a deformed discrete difference operator on $\psi_i$ and $\psi_j^*$:
\[
\Sigma_{(\bal)} \psi_i \Sigma_{(\bal)}^{-1} = \psi_{i+1} + \alpha_i \psi_i,
\qquad\qquad
\Sigma_{(\bal)}^{-1} \psi_i^* \Sigma_{(\bal)} = \psi_{i-1}^* + \alpha_i \psi_i^*.
\]
These deformed discrete difference operators were used in~\cite[Eq.~(4.6)]{MS20} in the context of refined dual Grothendieck polynomials.
It would be interesting to see if there is an explicit connection between this and the free fermionic construction in~\cite{IMS24}; \textit{cf.}~\cite{Assiotis23,IMS23}.

Next, under the boson-fermion correspondence from Theorem~\ref{thm:boson_fermion}, the deformed current operators act on the space of double supersymmetric functions via
\[
J_k^{(\alpha)}\cdot f :=
\begin{cases} p_kf, & \text{if } k<0, \\
0 & \text{if } k = 0, \\
\frac{\partial f}{\partial p_k}, & \text{if } k>0.
\end{cases}
\]
Subsequently, the (deformed) Murnaghan--Nakayama rule~\cite[Thm.~5.23]{BHS} (and its ``dual'') always yields \emph{finite} sums when $\bbe = 0$.
As a consequence, we see that every symmetric function can be written as a \emph{finite} sum of double Schur functions.
Following~\cite{BHS}, define $\pp = (p_1, p_2, \ldots)$ as the generators of the polynomial ring $\ZZ[\bal][\pp]$, which we will generally consider under the specialization $p_k = p_k(\xx/\yy)$.
As such, we write $\pp$ as the inputs for the corresponding symmetric functions instead of $\xx/\yy$.
Likewise, we can write the supersymmetric functions in Section~\ref{sec:double_funcs} in terms of $\pp$ instead of $\xx/\yy$.

\begin{example}
\label{ex:jkexample_beta_zero}
Let $\lambda = (8,3,1)$, and consider the ket $\ket{\lambda} := \ket{\lambda}_0 = v_8\wedge v_2\wedge v_{-1}\wedge v_{-3}\wedge v_{-4}\wedge\cdots$. This vector can be represented as the following particle diagram. Particles (black circles) are placed at positions $i$ such that $v_i$ appears in $|\lambda\rangle$, while holes (white circles) are placed at the remaining positions.
\[
\scalebox{0.8}{\begin{tikzpicture}
  \foreach \i in {-4,-3,...,12}
    \draw[fill=white] (\i,0) circle (.3);
  \foreach \i in {12,11,9,6,0}
    \draw[fill=black] (\i,0) circle (.3);
  \foreach \i in {-4,-3,...,12}
    \node at (8-\i,-.6) {$\i$};
  \node at (-5,0) {$\cdots$};
  \node at (13,0) {$\cdots$};
\end{tikzpicture}}
\]

First consider the action of $J_3^{(\bal)}$ on $\ket{\lambda}$.
This action corresponds to moving a single particle at least three spaces to the right.
If the particle started in spot $j$ and ended in spot $i$, the resulting coefficient is $A_{ij}^3$ times a sign corresponding to how many particles lie in between $i$ and $j$.
This results in the following diagrams, with coefficients displayed to their right, and the action on partitions to the right.
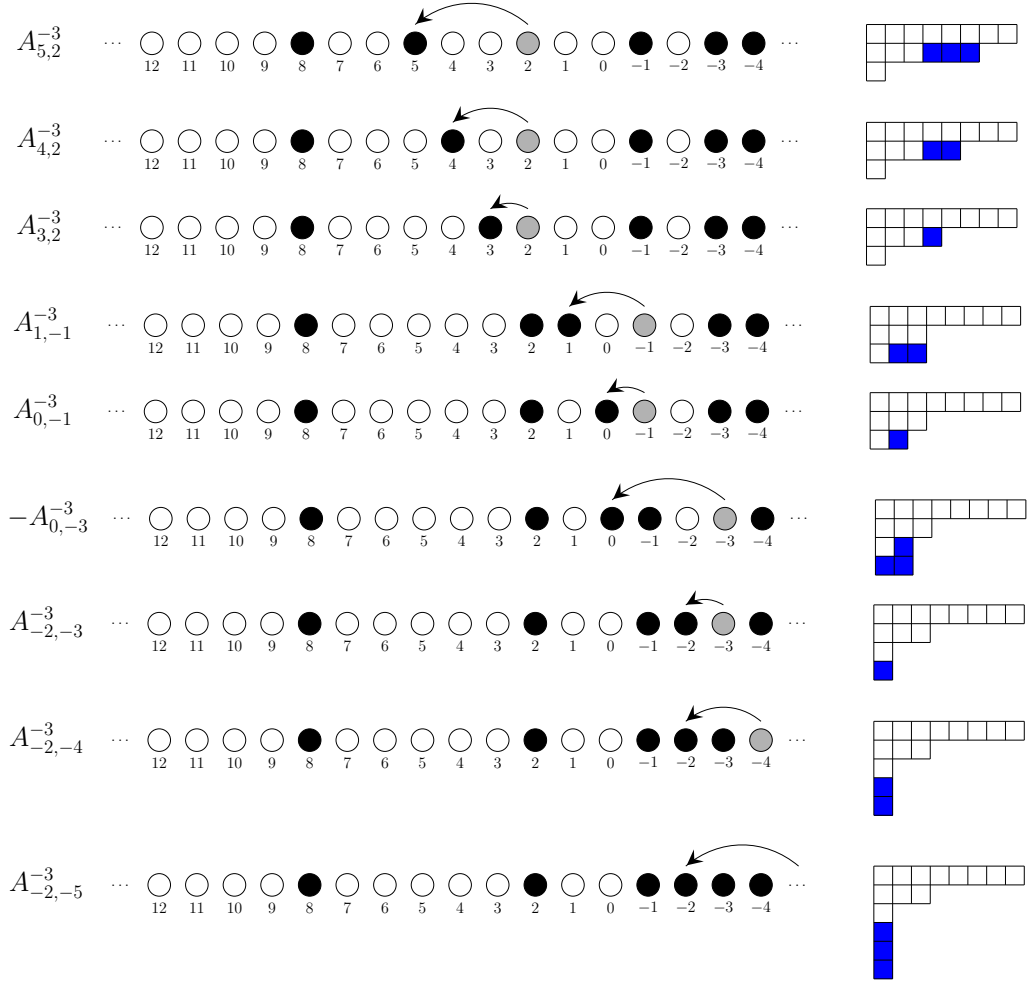
\[\scalebox{0.5}{\begin{tikzpicture}
  \draw[fill=black] (12,0) circle (.3);
  \draw[fill=black] (11,0) circle (.3);
  \draw[fill=white] (10,0) circle (.3);
  \draw[fill=black] (9,0) circle (.3);
  \draw[fill=white] (8,0) circle (.3);
  \draw[fill=white] (7,0) circle (.3);
  \draw[fill=black] (6,0) circle (.3);
  \draw[fill=white] (5,0) circle (.3);
  \draw[fill=white] (4,0) circle (.3);
  \draw[fill=black] (3,0) circle (.3);
  \draw[fill=white] (2,0) circle (.3);
  \draw[fill=white] (1,0) circle (.3);
  \draw[fill=black!30]  (0,0) circle (.3);
  \draw[fill=white] (-1,0) circle (.3);
  \draw[fill=white] (-2,0) circle (.3);
  \draw[fill=white] (-3,0) circle (.3);
  \draw[fill=white] (-4,0) circle (.3);
  \foreach \i in {-4,-3,...,12}
    \node at (8-\i,-.6) {$\i$};
  \node at (-5,0) {\ldots};
  \node at (13,0) {\ldots};
  \draw \truearrow (0,0.5) to [out=25,in=155] (3,0.5);
  \node at (-7,0) {{\LARGE $A_{5,8}^3$}};
  
  \draw (15,0.5)--(15,-1);
  \draw (15.5,0.5)--(15.5,-1);
  \draw (16,0.5)--(16,-0.5);
  \draw (16.5,0.5)--(16.5,-0.5);
  \draw (17,0.5)--(17,0);
  \draw (17.5,0.5)--(17.5,0);
  \draw (18,0.5)--(18,0);
  \draw (18.5,0.5)--(18.5,0);
  \draw (19,0.5)--(19,0);
  \draw (15,0.5)--(19,0.5);
  \draw (15,0)--(19,0);
  \draw (15,-0.5)--(16.5,-0.5);
  \draw (15,-1)--(15.5,-1);
  \fill[red, opacity=0.5] (17.5,0) rectangle (19,0.5);
\end{tikzpicture}}\]
\[\scalebox{0.5}{\begin{tikzpicture}
  \draw[fill=black] (12,0) circle (.3);
  \draw[fill=black] (11,0) circle (.3);
  \draw[fill=white] (10,0) circle (.3);
  \draw[fill=black] (9,0) circle (.3);
  \draw[fill=white] (8,0) circle (.3);
  \draw[fill=white] (7,0) circle (.3);
  \draw[fill=black] (6,0) circle (.3);
  \draw[fill=white] (5,0) circle (.3);
  \draw[fill=black] (4,0) circle (.3);
  \draw[fill=white] (3,0) circle (.3);
  \draw[fill=white] (2,0) circle (.3);
  \draw[fill=white] (1,0) circle (.3);
  \draw[fill=black!30] (0,0) circle (.3);
  \draw[fill=white] (-1,0) circle (.3);
  \draw[fill=white] (-2,0) circle (.3);
  \draw[fill=white] (-3,0) circle (.3);
  \draw[fill=white] (-4,0) circle (.3);
  \foreach \i in {-4,-3,...,12}
    \node at (8-\i,-.6) {$\i$};
  \node at (-5,0) {\ldots};
  \node at (13,0) {\ldots};
  \draw \truearrow (0,0.6) to [out=25,in=155] (4.0,0.6);
  \node at (-7,0) {{\LARGE $A_{4,8}^3$}};
  
  \draw (15,0.5)--(15,-1);
  \draw (15.5,0.5)--(15.5,-1);
  \draw (16,0.5)--(16,-0.5);
  \draw (16.5,0.5)--(16.5,-0.5);
  \draw (17,0.5)--(17,0);
  \draw (17.5,0.5)--(17.5,0);
  \draw (18,0.5)--(18,0);
  \draw (18.5,0.5)--(18.5,0);
  \draw (19,0.5)--(19,0);
  \draw (15,0.5)--(19,0.5);
  \draw (15,0)--(19,0);
  \draw (15,-0.5)--(16.5,-0.5);
  \draw (15,-1)--(15.5,-1);
  \fill[red, opacity=0.5] (17,0) rectangle (19,0.5);
\end{tikzpicture}}\]
\[\scalebox{0.5}{\begin{tikzpicture}
  \draw[fill=black] (12,0) circle (.3);
  \draw[fill=black] (11,0) circle (.3);
  \draw[fill=white] (10,0) circle (.3);
  \draw[fill=black] (9,0) circle (.3);
  \draw[fill=white] (8,0) circle (.3);
  \draw[fill=white] (7,0) circle (.3);
  \draw[fill=black] (6,0) circle (.3);
  \draw[fill=black] (5,0) circle (.3);
  \draw[fill=white] (4,0) circle (.3);
  \draw[fill=white] (3,0) circle (.3);
  \draw[fill=white] (2,0) circle (.3);
  \draw[fill=white] (1,0) circle (.3);
  \draw[fill=black!30] (0,0) circle (.3);
  \draw[fill=white] (-1,0) circle (.3);
  \draw[fill=white] (-2,0) circle (.3);
  \draw[fill=white] (-3,0) circle (.3);
  \draw[fill=white] (-4,0) circle (.3);
  \foreach \i in {-4,-3,...,12}
    \node at (8-\i,-.6) {$\i$};
  \node at (-5,0) {\ldots};
  \node at (13,0) {\ldots};
  \draw \truearrow (0,0.5) to [out=25,in=155] (5,0.5);
  \node at (-7,0) {{\LARGE $A_{3,8}^3$}};
  
  \draw (15,0.5)--(15,-1);
  \draw (15.5,0.5)--(15.5,-1);
  \draw (16,0.5)--(16,-0.5);
  \draw (16.5,0.5)--(16.5,-0.5);
  \draw (17,0.5)--(17,0);
  \draw (17.5,0.5)--(17.5,0);
  \draw (18,0.5)--(18,0);
  \draw (18.5,0.5)--(18.5,0);
  \draw (19,0.5)--(19,0);
  \draw (15,0.5)--(19,0.5);
  \draw (15,0)--(19,0);
  \draw (15,-0.5)--(16.5,-0.5);
  \draw (15,-1)--(15.5,-1);
  \fill[red, opacity=0.5] (16.5,0) rectangle (19,0.5);
\end{tikzpicture}}\]
\[\scalebox{0.5}{\begin{tikzpicture}
  \draw[fill=black] (12,0) circle (.3);
  \draw[fill=black] (11,0) circle (.3);
  \draw[fill=white] (10,0) circle (.3);
  \draw[fill=black] (9,0) circle (.3);
  \draw[fill=white] (8,0) circle (.3);
  \draw[fill=black] (7,0) circle (.3);
  \draw[fill=black] (6,0) circle (.3);
  \draw[fill=white] (5,0) circle (.3);
  \draw[fill=white] (4,0) circle (.3);
  \draw[fill=white] (3,0) circle (.3);
  \draw[fill=white] (2,0) circle (.3);
  \draw[fill=white] (1,0) circle (.3);
  \draw[fill=black!30] (0,0) circle (.3);
  \draw[fill=white] (-1,0) circle (.3);
  \draw[fill=white] (-2,0) circle (.3);
  \draw[fill=white] (-3,0) circle (.3);
  \draw[fill=white] (-4,0) circle (.3);
  \foreach \i in {-4,-3,...,12}
    \node at (8-\i,-.6) {$\i$};
  \node at (-5,0) {\ldots};
  \node at (13,0) {\ldots};
  \draw \truearrow (0,0.5) to [out=25,in=155] (7,0.5);
  \node at (-7,0) {{\LARGE $-A_{1,8}^3$}};
    
  \draw (15,0.5)--(15,-1);
  \draw (15.5,0.5)--(15.5,-1);
  \draw (16,0.5)--(16,-0.5);
  \draw (16.5,0.5)--(16.5,-0.5);
  \draw (17,0.5)--(17,0);
  \draw (17.5,0.5)--(17.5,0);
  \draw (18,0.5)--(18,0);
  \draw (18.5,0.5)--(18.5,0);
  \draw (19,0.5)--(19,0);
  \draw (15,0.5)--(19,0.5);
  \draw (15,0)--(19,0);
  \draw (15,-0.5)--(16.5,-0.5);
  \draw (15,-1)--(15.5,-1);
  \fill[red, opacity=0.5] (16,0) rectangle (19,0.5);
  \fill[red, opacity=0.5] (16,-0.5) rectangle (16.5,0);
\end{tikzpicture}}\]
\[\scalebox{0.5}{\begin{tikzpicture}
  \draw[fill=black] (12,0) circle (.3);
  \draw[fill=black] (11,0) circle (.3);
  \draw[fill=white] (10,0) circle (.3);
  \draw[fill=black] (9,0) circle (.3);
  \draw[fill=black] (8,0) circle (.3);
  \draw[fill=white] (7,0) circle (.3);
  \draw[fill=black] (6,0) circle (.3);
  \draw[fill=white] (5,0) circle (.3);
  \draw[fill=white] (4,0) circle (.3);
  \draw[fill=white] (3,0) circle (.3);
  \draw[fill=white] (2,0) circle (.3);
  \draw[fill=white] (1,0) circle (.3);
  \draw[fill=black!30] (0,0) circle (.3);
  \draw[fill=white] (-1,0) circle (.3);
  \draw[fill=white] (-2,0) circle (.3);
  \draw[fill=white] (-3,0) circle (.3);
  \draw[fill=white] (-4,0) circle (.3);
  \foreach \i in {-4,-3,...,12}
    \node at (8-\i,-.6) {$\i$};
  \node at (-5,0) {\ldots};
  \node at (13,0) {\ldots};
  \draw \truearrow (0,0.5) to [out=25,in=155] (8,0.5);
  \node at (-7,0) {{\LARGE $-A_{0,8}^3$}};
  
  \draw (15,0.5)--(15,-1);
  \draw (15.5,0.5)--(15.5,-1);
  \draw (16,0.5)--(16,-0.5);
  \draw (16.5,0.5)--(16.5,-0.5);
  \draw (17,0.5)--(17,0);
  \draw (17.5,0.5)--(17.5,0);
  \draw (18,0.5)--(18,0);
  \draw (18.5,0.5)--(18.5,0);
  \draw (19,0.5)--(19,0);
  \draw (15,0.5)--(19,0.5);
  \draw (15,0)--(19,0);
  \draw (15,-0.5)--(16.5,-0.5);
  \draw (15,-1)--(15.5,-1);
  \fill[red, opacity=0.5] (16,0) rectangle (19,0.5);
  \fill[red, opacity=0.5] (15.5,-0.5) rectangle (16.5,0);
\end{tikzpicture}}\]
\[\scalebox{0.5}{\begin{tikzpicture}
  \draw[fill=black] (12,0) circle (.3);
  \draw[fill=black] (11,0) circle (.3);
  \draw[fill=black] (10,0) circle (.3);
  \draw[fill=black] (9,0) circle (.3);
  \draw[fill=white] (8,0) circle (.3);
  \draw[fill=white] (7,0) circle (.3);
  \draw[fill=black] (6,0) circle (.3);
  \draw[fill=white] (5,0) circle (.3);
  \draw[fill=white] (4,0) circle (.3);
  \draw[fill=white] (3,0) circle (.3);
  \draw[fill=white] (2,0) circle (.3);
  \draw[fill=white] (1,0) circle (.3);
  \draw[fill=black!30] (0,0) circle (.3);
  \draw[fill=white] (-1,0) circle (.3);
  \draw[fill=white] (-2,0) circle (.3);
  \draw[fill=white] (-3,0) circle (.3);
  \draw[fill=white] (-4,0) circle (.3);
  \foreach \i in {-4,-3,...,12}
    \node at (8-\i,-.6) {$\i$};
  \node at (-5,0) {\ldots};
  \node at (13,0) {\ldots};
  \draw \truearrow (0,0.5) to [out=25,in=155] (10,0.5);
  \node at (-7,0) {{\LARGE $A_{-2,8}^3$}};
  
  \draw (15,0.5)--(15,-1);
  \draw (15.5,0.5)--(15.5,-1);
  \draw (16,0.5)--(16,-0.5);
  \draw (16.5,0.5)--(16.5,-0.5);
  \draw (17,0.5)--(17,0);
  \draw (17.5,0.5)--(17.5,0);
  \draw (18,0.5)--(18,0);
  \draw (18.5,0.5)--(18.5,0);
  \draw (19,0.5)--(19,0);
  \draw (15,0.5)--(19,0.5);
  \draw (15,0)--(19,0);
  \draw (15,-0.5)--(16.5,-0.5);
  \draw (15,-1)--(15.5,-1);
  \fill[red, opacity=0.5] (16,0) rectangle (19,0.5);
  \fill[red, opacity=0.5] (15,-0.5) rectangle (16.5,0);
  \fill[red, opacity=0.5] (15,-1) rectangle (15.5,-0.5);
\end{tikzpicture}}\]
\[\scalebox{0.5}{\begin{tikzpicture}
  \draw[fill=black] (12,0) circle (.3);
  \draw[fill=black] (11,0) circle (.3);
  \draw[fill=black] (10,0) circle (.3);
  \draw[fill=black] (9,0) circle (.3);
  \draw[fill=white] (8,0) circle (.3);
  \draw[fill=white] (7,0) circle (.3);
  \draw[fill=black!30] (6,0) circle (.3);
  \draw[fill=white] (5,0) circle (.3);
  \draw[fill=white] (4,0) circle (.3);
  \draw[fill=white] (3,0) circle (.3);
  \draw[fill=white] (2,0) circle (.3);
  \draw[fill=white] (1,0) circle (.3);
  \draw[fill=black] (0,0) circle (.3);
  \draw[fill=white] (-1,0) circle (.3);
  \draw[fill=white] (-2,0) circle (.3);
  \draw[fill=white] (-3,0) circle (.3);
  \draw[fill=white] (-4,0) circle (.3);
  \foreach \i in {-4,-3,...,12}
    \node at (8-\i,-.6) {$\i$};
  \node at (-5,0) {\ldots};
  \node at (13,0) {\ldots};
  \draw \truearrow (6,0.5) to [out=25,in=155] (10,0.5);
  \node at (-7,0) {{\LARGE $-A_{-2,2}^3$}};
    
  \draw (15,0.5)--(15,-1);
  \draw (15.5,0.5)--(15.5,-1);
  \draw (16,0.5)--(16,-0.5);
  \draw (16.5,0.5)--(16.5,-0.5);
  \draw (17,0.5)--(17,0);
  \draw (17.5,0.5)--(17.5,0);
  \draw (18,0.5)--(18,0);
  \draw (18.5,0.5)--(18.5,0);
  \draw (19,0.5)--(19,0);
  \draw (15,0.5)--(19,0.5);
  \draw (15,0)--(19,0);
  \draw (15,-0.5)--(16.5,-0.5);
  \draw (15,-1)--(15.5,-1);
  \fill[red, opacity=0.5] (15,-0.5) rectangle (16.5,0);
  \fill[red, opacity=0.5] (15,-1) rectangle (15.5,-0.5);
\end{tikzpicture}}\]
Therefore, we obtain
\begin{align*}
J_3^{(\bal)} \ket{(8,3,1)} &= \sum_{i,j} A_{ij}^3 E_{i,j} (v_8\wedge v_2\wedge v_{-1}\wedge v_{-3}\wedge v_{-4}\wedge\cdots)
\\ &= A^3_{5,8} \ket{(5,3,1)} + A^3_{4,8}\ket{(4,3,1)} + A^3_{3,8}\ket{(3,3,1)} - A^3_{1,8}\ket{(2,2,1)}
\\ & \hspace{20pt} - A^3_{0,8}\ket{(2,1,1)} + A^3_{-2,8}\ket{(2)} - A^3_{-2,2}\ket{(8)}.
\end{align*}

The partitions $(5,3,1)$, $(4,3,1)$, $(3,3,1)$, $(2,2,1)$, $(2,1,1)$, $(2)$, and $(8)$ are precisely those $\mu$ such that $\lambda/\mu$ is a ribbon of size at least 3.
In addition, the sign on each term is positive when the height of the ribbon is odd, and negative when it is even.

Next, we consider the action of $J_{-3}^{(\bal)}$ on $\ket{\lambda}$.
We start with the coefficients for $\ket{\mu}$ with $\mu\ne\lambda$.
The action of $J_{-3}^{(\bal)}$ corresponds to moving a single particle at most three spaces to the left.
If the particle started in spot $i$ and ended in spot $j$, the resulting coefficient is $A_{j,i}^3$ times a sign corresponding to how many particles lie in between $i$ and $j$.
The resulting diagrams are as follows.
\[\scalebox{0.5}{\begin{tikzpicture}
  \draw[fill=black] (12,0) circle (.3);
  \draw[fill=black] (11,0) circle (.3);
  \draw[fill=white] (10,0) circle (.3);
  \draw[fill=black] (9,0) circle (.3);
  \draw[fill=white] (8,0) circle (.3);
  \draw[fill=white] (7,0) circle (.3);
  \draw[fill=black] (6,0) circle (.3);
  \draw[fill=white] (5,0) circle (.3);
  \draw[fill=white] (4,0) circle (.3);
  \draw[fill=white] (3,0) circle (.3);
  \draw[fill=white] (2,0) circle (.3);
  \draw[fill=white] (1,0) circle (.3);
  \draw[fill=black!30] (0,0) circle (.3);
  \draw[fill=white] (-1,0) circle (.3);
  \draw[fill=white] (-2,0) circle (.3);
  \draw[fill=black] (-3,0) circle (.3);
  \draw[fill=white] (-4,0) circle (.3);
  \foreach \i in {-4,...,12}
    \node at (8-\i,-0.6) {$\i$};
  \node at (-5,0) {\ldots};
  \node at (13,0) {\ldots};
  \draw \truearrow (0,0.5) to [out=140,in=40] (-3,0.5);
  \node at (-7,0) {{\LARGE $A_{11,8}^{-3}$}};

  \fill[blue, opacity=0.5] (19,0) rectangle (20.5,0.5);
  \draw (15,0.5)--(15,-1);
  \draw (15.5,0.5)--(15.5,-1);
  \draw (16,0.5)--(16,-0.5);
  \draw (16.5,0.5)--(16.5,-0.5);
  \draw (17,0.5)--(17,0);
  \draw (17.5,0.5)--(17.5,0);
  \draw (18,0.5)--(18,0);
  \draw (18.5,0.5)--(18.5,0);
  \draw (19,0.5)--(19,0);
  \draw (19.5,0.5)--(19.5,0);
  \draw (20,0.5)--(20,0);
  \draw (20.5,0.5)--(20.5,0);
  \draw (15,0.5)--(20.5,0.5);
  \draw (15,0)--(20.5,0);
  \draw (15,-0.5)--(16.5,-0.5);
  \draw (15,-1)--(15.5,-1);
  \node at (20.5,0) {};
\end{tikzpicture}}\]
\[\scalebox{0.5}{\begin{tikzpicture}
  \draw[fill=black] (12,0) circle (.3);
  \draw[fill=black] (11,0) circle (.3);
  \draw[fill=white] (10,0) circle (.3);
  \draw[fill=black] (9,0) circle (.3);
  \draw[fill=white] (8,0) circle (.3);
  \draw[fill=white] (7,0) circle (.3);
  \draw[fill=black] (6,0) circle (.3);
  \draw[fill=white] (5,0) circle (.3);
  \draw[fill=white] (4,0) circle (.3);
  \draw[fill=white] (3,0) circle (.3);
  \draw[fill=white] (2,0) circle (.3);
  \draw[fill=white] (1,0) circle (.3);
  \draw[fill=black!30] (0,0) circle (.3);
  \draw[fill=white] (-1,0) circle (.3);
  \draw[fill=black] (-2,0) circle (.3);
  \draw[fill=white] (-3,0) circle (.3);
  \draw[fill=white] (-4,0) circle (.3);
  \foreach \i in {-4,...,12}
    \node at (8-\i,-0.6) {$\i$};
  \node at (-5,0) {\ldots};
  \node at (13,0) {\ldots};
  \draw \truearrow (0,0.5) to [out=140,in=40] (-2,0.5);
  \node at (-7,0) {{\LARGE $A_{10,8}^{-3}$}};

  \fill[blue, opacity=0.5] (19,0) rectangle (20,0.5);
  \draw (15,0.5)--(15,-1);
  \draw (15.5,0.5)--(15.5,-1);
  \draw (16,0.5)--(16,-0.5);
  \draw (16.5,0.5)--(16.5,-0.5);
  \draw (17,0.5)--(17,0);
  \draw (17.5,0.5)--(17.5,0);
  \draw (18,0.5)--(18,0);
  \draw (18.5,0.5)--(18.5,0);
  \draw (19,0.5)--(19,0);
  \draw (19.5,0.5)--(19.5,0);
  \draw (20,0.5)--(20,0);
  \draw (15,0.5)--(20,0.5);
  \draw (15,0)--(20,0);
  \draw (15,-0.5)--(16.5,-0.5);
  \draw (15,-1)--(15.5,-1);
  \node at (20.5,0) {};
\end{tikzpicture}}\]
\[\scalebox{0.5}{\begin{tikzpicture}
  \draw[fill=black] (12,0) circle (.3);
  \draw[fill=black] (11,0) circle (.3);
  \draw[fill=white] (10,0) circle (.3);
  \draw[fill=black] (9,0) circle (.3);
  \draw[fill=white] (8,0) circle (.3);
  \draw[fill=white] (7,0) circle (.3);
  \draw[fill=black] (6,0) circle (.3);
  \draw[fill=white] (5,0) circle (.3);
  \draw[fill=white] (4,0) circle (.3);
  \draw[fill=white] (3,0) circle (.3);
  \draw[fill=white] (2,0) circle (.3);
  \draw[fill=white] (1,0) circle (.3);
  \draw[fill=black!30] (0,0) circle (.3);
  \draw[fill=black] (-1,0) circle (.3);
  \draw[fill=white] (-2,0) circle (.3);
  \draw[fill=white] (-3,0) circle (.3);
  \draw[fill=white] (-4,0) circle (.3);
  \foreach \i in {-4,...,12}
    \node at (8-\i,-0.6) {$\i$};
  \node at (-5,0) {\ldots};
  \node at (13,0) {\ldots};
  \draw \truearrow (0,0.5) to [out=150,in=30] (-1,0.5);
  \node at (-7,0) {{\LARGE $A_{9,8}^{-3}$}};

  \fill[blue, opacity=0.5] (19,0) rectangle (19.5,0.5);
  \draw (15,0.5)--(15,-1);
  \draw (15.5,0.5)--(15.5,-1);
  \draw (16,0.5)--(16,-0.5);
  \draw (16.5,0.5)--(16.5,-0.5);
  \draw (17,0.5)--(17,0);
  \draw (17.5,0.5)--(17.5,0);
  \draw (18,0.5)--(18,0);
  \draw (18.5,0.5)--(18.5,0);
  \draw (19,0.5)--(19,0);
  \draw (19.5,0.5)--(19.5,0);
  \draw (15,0.5)--(19.5,0.5);
  \draw (15,0)--(19.5,0);
  \draw (15,-0.5)--(16.5,-0.5);
  \draw (15,-1)--(15.5,-1);
  \node at (20.5,0) {};
\end{tikzpicture}}\]
\[\scalebox{0.5}{\begin{tikzpicture}
  \draw[fill=black] (12,0) circle (.3);
  \draw[fill=black] (11,0) circle (.3);
  \draw[fill=white] (10,0) circle (.3);
  \draw[fill=black] (9,0) circle (.3);
  \draw[fill=white] (8,0) circle (.3);
  \draw[fill=white] (7,0) circle (.3);
  \draw[fill=black!30] (6,0) circle (.3);
  \draw[fill=white] (5,0) circle (.3);
  \draw[fill=white] (4,0) circle (.3);
  \draw[fill=black] (3,0) circle (.3);
  \draw[fill=white] (2,0) circle (.3);
  \draw[fill=white] (1,0) circle (.3);
  \draw[fill=black] (0,0) circle (.3);
  \draw[fill=white] (-1,0) circle (.3);
  \draw[fill=white] (-2,0) circle (.3);
  \draw[fill=white] (-3,0) circle (.3);
  \draw[fill=white] (-4,0) circle (.3);
  \foreach \i in {-4,...,12}
    \node at (8-\i,-0.6) {$\i$};
  \node at (-5,0) {\ldots};
  \node at (13,0) {\ldots};
  \draw \truearrow (6,0.5) to [out=140,in=40] (3,0.5);
  \node at (-7,0) {{\LARGE $A_{5,2}^{-3}$}};

  \fill[blue, opacity=0.5] (16.5,-0.5) rectangle (18,0);
  \draw (15,0.5)--(15,-1);
  \draw (15.5,0.5)--(15.5,-1);
  \draw (16,0.5)--(16,-0.5);
  \draw (16.5,0.5)--(16.5,-0.5);
  \draw (17,0.5)--(17,-0.5);
  \draw (17.5,0.5)--(17.5,-0.5);
  \draw (18,0.5)--(18,-0.5);
  \draw (18.5,0.5)--(18.5,0);
  \draw (19,0.5)--(19,0);
  \draw (15,0.5)--(19,0.5);
  \draw (15,0)--(19,0);
  \draw (15,-0.5)--(18,-0.5);
  \draw (15,-1)--(15.5,-1);
  \node at (20.5,0) {};
\end{tikzpicture}}\]
\[\scalebox{0.5}{\begin{tikzpicture}
  \draw[fill=black] (12,0) circle (.3);
  \draw[fill=black] (11,0) circle (.3);
  \draw[fill=white] (10,0) circle (.3);
  \draw[fill=black] (9,0) circle (.3);
  \draw[fill=white] (8,0) circle (.3);
  \draw[fill=white] (7,0) circle (.3);
  \draw[fill=black!30] (6,0) circle (.3);
  \draw[fill=white] (5,0) circle (.3);
  \draw[fill=black] (4,0) circle (.3);
  \draw[fill=white] (3,0) circle (.3);
  \draw[fill=white] (2,0) circle (.3);
  \draw[fill=white] (1,0) circle (.3);
  \draw[fill=black] (0,0) circle (.3);
  \draw[fill=white] (-1,0) circle (.3);
  \draw[fill=white] (-2,0) circle (.3);
  \draw[fill=white] (-3,0) circle (.3);
  \draw[fill=white] (-4,0) circle (.3);
  \foreach \i in {-4,...,12}
    \node at (8-\i,-0.6) {$\i$};
  \node at (-5,0) {\ldots};
  \node at (13,0) {\ldots};
  \draw \truearrow (6,0.5) to [out=140,in=40] (4,0.5);
  \node at (-7,0) {{\LARGE $A_{4,2}^{-3}$}};

  \fill[blue, opacity=0.5] (16.5,-0.5) rectangle (17.5,0);
  \draw (15,0.5)--(15,-1);
  \draw (15.5,0.5)--(15.5,-1);
  \draw (16,0.5)--(16,-0.5);
  \draw (16.5,0.5)--(16.5,-0.5);
  \draw (17,0.5)--(17,-0.5);
  \draw (17.5,0.5)--(17.5,-0.5);
  \draw (18,0.5)--(18,0);
  \draw (18.5,0.5)--(18.5,0);
  \draw (19,0.5)--(19,0);
  \draw (15,0.5)--(19,0.5);
  \draw (15,0)--(19,0);
  \draw (15,-0.5)--(17.5,-0.5);
  \draw (15,-1)--(15.5,-1);
  \node at (20.5,0) {};
\end{tikzpicture}}\]
\[\scalebox{0.5}{\begin{tikzpicture}
  \draw[fill=black] (12,0) circle (.3);
  \draw[fill=black] (11,0) circle (.3);
  \draw[fill=white] (10,0) circle (.3);
  \draw[fill=black] (9,0) circle (.3);
  \draw[fill=white] (8,0) circle (.3);
  \draw[fill=white] (7,0) circle (.3);
  \draw[fill=black!30] (6,0) circle (.3);
  \draw[fill=black] (5,0) circle (.3);
  \draw[fill=white] (4,0) circle (.3);
  \draw[fill=white] (3,0) circle (.3);
  \draw[fill=white] (2,0) circle (.3);
  \draw[fill=white] (1,0) circle (.3);
  \draw[fill=black] (0,0) circle (.3);
  \draw[fill=white] (-1,0) circle (.3);
  \draw[fill=white] (-2,0) circle (.3);
  \draw[fill=white] (-3,0) circle (.3);
  \draw[fill=white] (-4,0) circle (.3);
  \foreach \i in {-4,...,12}
    \node at (8-\i,-0.6) {$\i$};
  \node at (-5,0) {\ldots};
  \node at (13,0) {\ldots};
  \draw \truearrow (6,0.5) to [out=150,in=30] (5,0.5);
  \node at (-7,0) {{\LARGE $A_{3,2}^{-3}$}};

  \fill[blue, opacity=0.5] (16.5,-0.5) rectangle (17,0);
  \draw (15,0.5)--(15,-1);
  \draw (15.5,0.5)--(15.5,-1);
  \draw (16,0.5)--(16,-0.5);
  \draw (16.5,0.5)--(16.5,-0.5);
  \draw (17,0.5)--(17,-0.5);
  \draw (17.5,0.5)--(17.5,0);
  \draw (18,0.5)--(18,0);
  \draw (18.5,0.5)--(18.5,0);
  \draw (19,0.5)--(19,0);
  \draw (15,0.5)--(19,0.5);
  \draw (15,0)--(19,0);
  \draw (15,-0.5)--(17,-0.5);
  \draw (15,-1)--(15.5,-1);
  \node at (20.5,0) {};
\end{tikzpicture}}\]
\[\scalebox{0.5}{\begin{tikzpicture}
  \draw[fill=black] (12,0) circle (.3);
  \draw[fill=black] (11,0) circle (.3);
  \draw[fill=white] (10,0) circle (.3);
  \draw[fill=black!30] (9,0) circle (.3);
  \draw[fill=white] (8,0) circle (.3);
  \draw[fill=black] (7,0) circle (.3);
  \draw[fill=black] (6,0) circle (.3);
  \draw[fill=white] (5,0) circle (.3);
  \draw[fill=white] (4,0) circle (.3);
  \draw[fill=white] (3,0) circle (.3);
  \draw[fill=white] (2,0) circle (.3);
  \draw[fill=white] (1,0) circle (.3);
  \draw[fill=black] (0,0) circle (.3);
  \draw[fill=white] (-1,0) circle (.3);
  \draw[fill=white] (-2,0) circle (.3);
  \draw[fill=white] (-3,0) circle (.3);
  \draw[fill=white] (-4,0) circle (.3);
  \foreach \i in {-4,...,12}
    \node at (8-\i,-0.6) {$\i$};
  \node at (-5,0) {\ldots};
  \node at (13,0) {\ldots};
  \draw \truearrow (9,0.5) to [out=140,in=40] (7,0.5);
  \node at (-7,0) {{\LARGE $A_{1,-1}^{-3}$}};
  
  \fill[blue, opacity=0.5] (15.5,-1) rectangle (16.5,-0.5);
  \draw (15,0.5)--(15,-1);
  \draw (15.5,0.5)--(15.5,-1);
  \draw (16,0.5)--(16,-1);
  \draw (16.5,0.5)--(16.5,-1);
  \draw (17,0.5)--(17,0);
  \draw (17.5,0.5)--(17.5,0);
  \draw (18,0.5)--(18,0);
  \draw (18.5,0.5)--(18.5,0);
  \draw (19,0.5)--(19,0);
  \draw (15,0.5)--(19,0.5);
  \draw (15,0)--(19,0);
  \draw (15,-0.5)--(16.5,-0.5);
  \draw (15,-1)--(16.5,-1);
  \node at (20.5,0) {};
\end{tikzpicture}}\]
\[\scalebox{0.5}{\begin{tikzpicture}
  \draw[fill=black] (12,0) circle (.3);
  \draw[fill=black] (11,0) circle (.3);
  \draw[fill=white] (10,0) circle (.3);
  \draw[fill=black!30] (9,0) circle (.3);
  \draw[fill=black] (8,0) circle (.3);
  \draw[fill=white] (7,0) circle (.3);
  \draw[fill=black] (6,0) circle (.3);
  \draw[fill=white] (5,0) circle (.3);
  \draw[fill=white] (4,0) circle (.3);
  \draw[fill=white] (3,0) circle (.3);
  \draw[fill=white] (2,0) circle (.3);
  \draw[fill=white] (1,0) circle (.3);
  \draw[fill=black] (0,0) circle (.3);
  \draw[fill=white] (-1,0) circle (.3);
  \draw[fill=white] (-2,0) circle (.3);
  \draw[fill=white] (-3,0) circle (.3);
  \draw[fill=white] (-4,0) circle (.3);
  \foreach \i in {-4,...,12}
    \node at (8-\i,-0.6) {$\i$};
  \node at (-5,0) {\ldots};
  \node at (13,0) {\ldots};
  \draw \truearrow (9,0.5) to [out=150,in=30] (8,0.5);
  \node at (-7,0) {{\LARGE $A_{0,-1}^{-3}$}};

  \fill[blue, opacity=0.5] (15.5,-1) rectangle (16,-0.5);
  \draw (15,0.5)--(15,-1);
  \draw (15.5,0.5)--(15.5,-1);
  \draw (16,0.5)--(16,-1);
  \draw (16.5,0.5)--(16.5,-0.5);
  \draw (17,0.5)--(17,0);
  \draw (17.5,0.5)--(17.5,0);
  \draw (18,0.5)--(18,0);
  \draw (18.5,0.5)--(18.5,0);
  \draw (19,0.5)--(19,0);
  \draw (15,0.5)--(19,0.5);
  \draw (15,0)--(19,0);
  \draw (15,-0.5)--(16.5,-0.5);
  \draw (15,-1)--(16,-1);
  \node at (20.5,0) {};
\end{tikzpicture}}\]
\[\scalebox{0.5}{\begin{tikzpicture}
  \draw[fill=black] (12,0) circle (.3);
  \draw[fill=black!30] (11,0) circle (.3);
  \draw[fill=white] (10,0) circle (.3);
  \draw[fill=black] (9,0) circle (.3);
  \draw[fill=black] (8,0) circle (.3);
  \draw[fill=white] (7,0) circle (.3);
  \draw[fill=black] (6,0) circle (.3);
  \draw[fill=white] (5,0) circle (.3);
  \draw[fill=white] (4,0) circle (.3);
  \draw[fill=white] (3,0) circle (.3);
  \draw[fill=white] (2,0) circle (.3);
  \draw[fill=white] (1,0) circle (.3);
  \draw[fill=black] (0,0) circle (.3);
  \draw[fill=white] (-1,0) circle (.3);
  \draw[fill=white] (-2,0) circle (.3);
  \draw[fill=white] (-3,0) circle (.3);
  \draw[fill=white] (-4,0) circle (.3);
  \foreach \i in {-4,...,12}
    \node at (8-\i,-0.6) {$\i$};
  \node at (-5,0) {\ldots};
  \node at (13,0) {\ldots};
  \draw \truearrow (11,0.5) to [out=140,in=40] (8,0.5);
  \node at (-7,0) {{\LARGE $-A_{0,-3}^{-3}$}};

  \fill[blue, opacity=0.5] (15.5,-1) rectangle (16,-0.5);
  \fill[blue, opacity=0.5] (15,-1.5) rectangle (16,-1);
  \draw (15,0.5)--(15,-1.5);
  \draw (15.5,0.5)--(15.5,-1.5);
  \draw (16,0.5)--(16,-1.5);
  \draw (16.5,0.5)--(16.5,-0.5);
  \draw (17,0.5)--(17,0);
  \draw (17.5,0.5)--(17.5,0);
  \draw (18,0.5)--(18,0);
  \draw (18.5,0.5)--(18.5,0);
  \draw (19,0.5)--(19,0);
  \draw (15,0.5)--(19,0.5);
  \draw (15,0)--(19,0);
  \draw (15,-0.5)--(16.5,-0.5);
  \draw (15,-1)--(16,-1);
  \draw (15,-1.5)--(16,-1.5);
  \node at (20.5,0) {};
\end{tikzpicture}}\]
\[\scalebox{0.5}{\begin{tikzpicture}
  \draw[fill=black] (12,0) circle (.3);
  \draw[fill=black!30] (11,0) circle (.3);
  \draw[fill=black] (10,0) circle (.3);
  \draw[fill=black] (9,0) circle (.3);
  \draw[fill=white] (8,0) circle (.3);
  \draw[fill=white] (7,0) circle (.3);
  \draw[fill=black] (6,0) circle (.3);
  \draw[fill=white] (5,0) circle (.3);
  \draw[fill=white] (4,0) circle (.3);
  \draw[fill=white] (3,0) circle (.3);
  \draw[fill=white] (2,0) circle (.3);
  \draw[fill=white] (1,0) circle (.3);
  \draw[fill=black] (0,0) circle (.3);
  \draw[fill=white] (-1,0) circle (.3);
  \draw[fill=white] (-2,0) circle (.3);
  \draw[fill=white] (-3,0) circle (.3);
  \draw[fill=white] (-4,0) circle (.3);
  \foreach \i in {-4,...,12}
    \node at (8-\i,-0.6) {$\i$};
  \node at (-5,0) {\ldots};
  \node at (13,0) {\ldots};
  \draw \truearrow (11,0.5) to [out=150,in=30] (10,0.5);
  \node at (-7,0) {{\LARGE $A_{-2,-3}^{-3}$}};

  \fill[blue, opacity=0.5] (15,-1.5) rectangle (15.5,-1);
  \draw (15,0.5)--(15,-1.5);
  \draw (15.5,0.5)--(15.5,-1.5);
  \draw (16,0.5)--(16,-0.5);
  \draw (16.5,0.5)--(16.5,-0.5);
  \draw (17,0.5)--(17,0);
  \draw (17.5,0.5)--(17.5,0);
  \draw (18,0.5)--(18,0);
  \draw (18.5,0.5)--(18.5,0);
  \draw (19,0.5)--(19,0);
  \draw (15,0.5)--(19,0.5);
  \draw (15,0)--(19,0);
  \draw (15,-0.5)--(16.5,-0.5);
  \draw (15,-1)--(15.5,-1);
  \draw (15,-1.5)--(15.5,-1.5);
  \node at (20.5,0) {};
\end{tikzpicture}}\]
\[\scalebox{0.5}{\begin{tikzpicture}
  \draw[fill=black!30] (12,0) circle (.3);
  \draw[fill=black] (11,0) circle (.3);
  \draw[fill=black] (10,0) circle (.3);
  \draw[fill=black] (9,0) circle (.3);
  \draw[fill=white] (8,0) circle (.3);
  \draw[fill=white] (7,0) circle (.3);
  \draw[fill=black] (6,0) circle (.3);
  \draw[fill=white] (5,0) circle (.3);
  \draw[fill=white] (4,0) circle (.3);
  \draw[fill=white] (3,0) circle (.3);
  \draw[fill=white] (2,0) circle (.3);
  \draw[fill=white] (1,0) circle (.3);
  \draw[fill=black] (0,0) circle (.3);
  \draw[fill=white] (-1,0) circle (.3);
  \draw[fill=white] (-2,0) circle (.3);
  \draw[fill=white] (-3,0) circle (.3);
  \draw[fill=white] (-4,0) circle (.3);
  \foreach \i in {-4,...,12}
    \node at (8-\i,-0.6) {$\i$};
  \node at (-5,0) {\ldots};
  \node at (13,0) {\ldots};
  \draw \truearrow (12,0.5) to [out=140,in=40] (10,0.5);
  \node at (-7,0) {{\LARGE $A_{-2,-4}^{-3}$}};

  \fill[blue, opacity=0.5] (15,-2) rectangle (15.5,-1);
  \draw (15,0.5)--(15,-2);
  \draw (15.5,0.5)--(15.5,-2);
  \draw (16,0.5)--(16,-0.5);
  \draw (16.5,0.5)--(16.5,-0.5);
  \draw (17,0.5)--(17,0);
  \draw (17.5,0.5)--(17.5,0);
  \draw (18,0.5)--(18,0);
  \draw (18.5,0.5)--(18.5,0);
  \draw (19,0.5)--(19,0);
  \draw (15,0.5)--(19,0.5);
  \draw (15,0)--(19,0);
  \draw (15,-0.5)--(16.5,-0.5);
  \draw (15,-1)--(15.5,-1);
  \draw (15,-1.5)--(15.5,-1.5);
  \draw (15,-2)--(15.5,-2);
  \node at (20.5,0) {};
\end{tikzpicture}}\]
\[\scalebox{0.5}{\begin{tikzpicture}
  \draw[fill=black] (12,0) circle (.3);
  \draw[fill=black] (11,0) circle (.3);
  \draw[fill=black] (10,0) circle (.3);
  \draw[fill=black] (9,0) circle (.3);
  \draw[fill=white] (8,0) circle (.3);
  \draw[fill=white] (7,0) circle (.3);
  \draw[fill=black] (6,0) circle (.3);
  \draw[fill=white] (5,0) circle (.3);
  \draw[fill=white] (4,0) circle (.3);
  \draw[fill=white] (3,0) circle (.3);
  \draw[fill=white] (2,0) circle (.3);
  \draw[fill=white] (1,0) circle (.3);
  \draw[fill=black] (0,0) circle (.3);
  \draw[fill=white] (-1,0) circle (.3);
  \draw[fill=white] (-2,0) circle (.3);
  \draw[fill=white] (-3,0) circle (.3);
  \draw[fill=white] (-4,0) circle (.3);
  \foreach \i in {-4,...,12}
    \node at (8-\i,-0.6) {$\i$};
  \node at (-5,0) {\ldots};
  \node at (13,0) {\ldots};
  \draw \truearrow (13,0.5) to [out=140,in=40] (10,0.5);
  \node at (-7,0) {{\LARGE $A_{-2,-5}^{-3}$}};

  \fill[blue, opacity=0.5] (15,-2.5) rectangle (15.5,-1);
  \draw (15,0.5)--(15,-2.5);
  \draw (15.5,0.5)--(15.5,-2.5);
  \draw (16,0.5)--(16,-0.5);
  \draw (16.5,0.5)--(16.5,-0.5);
  \draw (17,0.5)--(17,0);
  \draw (17.5,0.5)--(17.5,0);
  \draw (18,0.5)--(18,0);
  \draw (18.5,0.5)--(18.5,0);
  \draw (19,0.5)--(19,0);
  \draw (15,0.5)--(19,0.5);
  \draw (15,0)--(19,0);
  \draw (15,-0.5)--(16.5,-0.5);
  \draw (15,-1)--(15.5,-1);
  \draw (15,-1.5)--(15.5,-1.5);
  \draw (15,-2)--(15.5,-2);
  \draw (15,-2.5)--(15.5,-2.5);
  \node at (20.5,0) {};
\end{tikzpicture}}\]
This determines the coefficient of $\ket{\mu}$ in the expansion of $J_{-k}^{(\bal)}\ket{\lambda}$, where $\lambda \ne \mu$.
For the coefficient of $\ket{\lambda}$, we need to take into account the projective representation~\eqref{eq:rhat}.
Hence, we compute it by only considering the particles (resp.\ holes) to the left (resp.\ right) of $\frac{1}{2}$ (\textit{cf}.~\cite[Eq.~(54)]{BHS}), which from~\eqref{eq:a_int} yields $\alpha_8^3 + \alpha_2^3 - \alpha_0^3 - \alpha_{-2}^3$.
Putting all this together, we obtain
\begin{align*}
J_{-3} \ket{(8,3,1)} &= A^{-3}_{11,8} \ket{(11,3,1)} + A^{-3}_{10,8} \ket{(10,3,1)} + A^{-3}_{9,8} \ket{(9,3,1)} + A^{-3}_{5,2} \ket{(8,6,1)}
\\ & \hspace{20pt} + A^{-3}_{4,2} \ket{(8,5,1)} + A^{-3}_{3,2} \ket{(8,4,1)} + A^{-3}_{1,-1} \ket{(8,3,3)} + A^{-3}_{0,-1} \ket{(8,3,2)}
\\ & \hspace{20pt} - A^{-3}_{0,-3}\ket{(8,3,2,2)} + A^{-3}_{-2,-3}\ket{(8,3,1,1)} - A^{-3}_{-2,-4}\ket{(8,3,1,1,1)}
\\ & \hspace{20pt} + A^{-3}_{-2,-5}\ket{(8,3,1,1,1,1)} + (\alpha_8^3 + \alpha_2^3 - \alpha_0^3 - \alpha_{-2}^3) \ket{(8,3,1)}.
\end{align*}
\end{example}

\begin{example}
Continuing Example~\ref{ex:jkexample_beta_zero}, applying the Murnaghan--Nakayama rule~\cite[Thm.~5.23]{BHS} with~\eqref{eq:a_int} yields
\begin{align*}
p_3 s_{(8,3,1)} &= s_{(11,3,1)} + (\alpha_8 + \alpha_9 + \alpha_{10}) s_{(10,3,1)} + (\alpha_8^2+\alpha_8\alpha_9+\alpha_9^2) s_{(9,3,1)} + s_{(8,6,1)}
\\ & \hspace{20pt} + (\alpha_2 + \alpha_3 + \alpha_4) s_{(8,5,1)} + (\alpha_3^2+\alpha_2\alpha_3+\alpha_3^2) s_{(8,4,1)}
\\ & \hspace{20pt} + (\alpha_{-1} + \alpha_0 + \alpha_1) s_{(8,3,3)} + (\alpha_{-1}^2 + \alpha_{-1} \alpha_0 + \alpha_0^2) s_{(8,3,2)}
\\ & \hspace{20pt} - s_{(8,3,2,2)} + (\alpha_{-3}^2 + \alpha_{-3} \alpha_{-2} + \alpha_{-2}^2) s_{(8,3,1,1)} -  (\alpha_{-4} + \alpha_{-3} + \alpha_{-2})  s_{(8,3,1,1,1)}
\\ & \hspace{20pt} + s_{(8,3,1,1,1,1)} + (\alpha_8^3 + \alpha_2^3 - \alpha_0^3 - \alpha_{-2}^3)s_{(8,3,1)}
\allowdisplaybreaks\\
3\frac{\partial s_{(8,3,1)}}{\partial p_3} &= s_{(5,3,1)} + e_1(-\bal_{(4,8)}) s_{(4,3,1)} + e_2(-\bal_{(3,8)}) s_{(3,3,1)}  - e_4(-\bal_{(1,8)}) s_{(2,2,1)}
\\ & \hspace{80pt}- e_5(-\bal_{(0,8)} )s_{(2,1,1)} + e_7(-\bal_{(-2,8)}) s_{(2)}  - e_1(-\bal_{(-2,2)}) s_{(8)},
\end{align*}
where we have written $s_{\lambda} = s_{\lambda}(\pp\dv\bal)$ for simplicity.
\end{example}

\begin{example}
\label{ex:powersum_expansion}
As another example, we can define the expansion of $p_k(\xx/\yy)$ in terms of the double Schur functions by taking $\lambda = \varnothing$ in the Murnaghan--Nakayama rule.
In particular,
\[
J_{-2}^{(\bal)} \ket{\varnothing} = \ket{(2,0)} - \ket{(1,1)} + (\alpha_0 + \alpha_1) \ket{(1,0)},
\]
which when we restrict to a single $\xx$ and $\yy$ variable, we obtain
\begin{align*}
x_1^2 - y_1^2 = p_2(\xx_1/\yy_1) & = h_2(\xx_1/\yy_1\dv\bal) - e_2(\xx_1/\yy_1\dv\bal) + (\alpha_0 + \alpha_1) h_1(\xx_1/\yy_1\dv\bal)
\\ & = \bigl((x_1 - \alpha_0)(x_1 - \alpha_1) + (y_1 + \alpha_0)(x_1 - \alpha_1) \bigr)
\\ & \hspace{20pt}  - \bigl( (y_1 + \alpha_1)(y_1 + \alpha_2) + (y_1 + \alpha_1)(x_1 - \alpha_2) \bigr)
\\ & \hspace{20pt} + (\alpha_0 + \alpha_1) \bigl( (x_1 - \alpha_0) + (y_1 + \alpha_0) \bigr).
\end{align*}
This can be readily verified by a direct computation.
\end{example}

Continuing from Example~\ref{ex:powersum_expansion}, by repeated applications of the Murnaghan--Nakayama rule, we obtain an expansion of $p_{\lambda}$ as a sum over double Schur functions.
More precisely, the coefficients in this expansion are given by a weighted (signed) sum over ribbon tableau analogous to the classical expansion of powersums into Schur functions.

\section{Raising operator formulas}
\label{sec:raising_operators}

We give an example of the integral formula~\cite[Eq.~(68)]{BHS}:
For a partition $\lambda$ and $\ell \geqslant \ell(\lambda)$, 
\begin{equation}
\label{eq:raising_operator_integral}
s_{\lambda}(\pp\dv\bal) = \sum_{k_1,\dotsc,k_{\ell}=0}^{\infty} \oint \prod_{i < j} \left(1 - \frac{z_i^{-1} - \alpha_{1-i}}{z_j^{-1} - \alpha_{1-i}} \right) \prod_{i=1}^{\ell} h_{k_i}(\pp\dv\sigma^{s_i}\bal) \frac{z_i^{-1} (z_i^{-1}|\sigma^{1-i} \bal)^{\lambda_i-1}}{(z_i^{-1}|\sigma^{s_i}\bal)^{k_i}}  \frac{dz_i}{2\pi\ii z_i}
\end{equation}
for some values $s_i$.
(That any $s_i$ is valid comes from the generating series, and we are also free to choose these values within any sum over all $k_i \in \ZZ_{\geqslant0}$.)
An important consequence of Proposition~\ref{prop:orthonormality_gen2} is that the sum in~\eqref{eq:raising_operator_integral} is finite, so we can interchange the summation and integral.
Define the shorthands
\[
h_{k,s} := h_k(\pp\dv\sigma^s\bal)
\qquad \text{ and } \qquad
h_{\lambda,\eta} := h_{\lambda_1,-\eta_1} \dotsm h_{\lambda_{\ell},-\eta_{\ell}}.
\]

\begin{example}
\label{ex:integral_raising}
Consider $\lambda$ with $\ell(\lambda) = 3$, and so we have
\begin{gather*}
\prod_{i<j} (1 - A_{ij}) = 1 - A_{12} - A_{13} - A_{23} + A_{12}A_{13} + A_{12}A_{23} + A_{13}A_{23} - A_{12}A_{13}A_{23},
\\
\text{where } \qquad
A_{12} := \frac{z_1^{-1} - \alpha_0}{z_2^{-1} - \alpha_0},
\qquad
A_{13} := \frac{z_1^{-1} - \alpha_0}{z_3^{-1} - \alpha_0},
\qquad
A_{23} := \frac{z_2^{-1} - \alpha_{-1}}{z_3^{-1} - \alpha_{-1}}.
\end{gather*}
Next, we use the linearity of the contour integral to compute each term using this expansion in~\eqref{eq:raising_operator_integral}.
Let $dF := \sum_{k_1,k_2,k_3=0}^{\infty} \frac{z_i^{-1} (z_i^{-1}|\sigma^{1-i} \bal)^{\lambda_i-1}}{(z_i^{-1}|\sigma^{1-i}\bal)^{k_i}}  \frac{dz_i}{2\pi\ii z_i}$, and so we compute
\begin{gather*}
\oint 1 \, dF = h_{(\lambda_1,\lambda_2,\lambda_3),(0,1,2)},
\oint A_{12} \, dF = h_{(\lambda_1+1,\lambda_2-1,\lambda_3),(1,0,2)},
\oint A_{23} \, dF = h_{(\lambda_1,\lambda_2+1,\lambda_3-1),(0,2,1)},
\allowdisplaybreaks\\
\oint A_{13} \, dF = h_{(\lambda_1+1,\lambda_2,\lambda_3-1),(1,1,1)} + (\alpha_0 - \alpha_{-1}) \sum_{k=2}^{\lambda_3-2} (\alpha_0|\sigma^{\lambda_3-2-k}\bal)^{k-2} h_{(\lambda_1+1,\lambda_2,\lambda_3-k),(1,1,1)},
 \allowdisplaybreaks\\
\oint A_{12} A_{13} \, dF = h_{(\lambda_1+2,\lambda_2-1,\lambda_3-1),(2,0,1)} + \cdots,
\qquad
\oint A_{12} A_{23} \, dF = h_{(\lambda_1+1,\lambda_2,\lambda_3-1),(1,1,1)} + \cdots,
\allowdisplaybreaks\\
\oint A_{13} A_{23} \, dF = h_{(\lambda_1+1,\lambda_2,\lambda_3-1),(1,1,1)} + \cdots,
\qquad
\oint A_{12} A_{13} A_{23} \, dF = h_{(\lambda_1+2,\lambda_2,\lambda_3-1),(2,1,0)} + \cdots,
\end{gather*}
where the trailing terms are multiple sums analogous to $\oint A_{13} \, dF$.
\end{example}

\begin{example}
While it might seem like all of the coefficients might always be similar to those of $\oint A_{13} \, dF$, this is not the case when $\ell(\lambda) \geqslant 4$.
In particular, we have
\begin{align*}
\oint (z - \alpha_0)^3 \sum_{k=0}^{\infty} h_k(\pp\dv\bal) \frac{(z^{-1}|\bal)^2}{(z^{-1}|\bal)} \frac{dz}{2\pi\ii z}
& = (\alpha_{-1} - \alpha_2) (\alpha_{-1} - \alpha_1) (\alpha_{-1} - \alpha_0)^3 h_{1,-2}
\\ &\hspace{20pt} + (\alpha_{-1} - \alpha_2) (\alpha_{-1} - \alpha_1) (\alpha_{-1} - \alpha_0)^2 h_{2,-2}
\\ &\hspace{20pt} + (\alpha_{-1} - \alpha_1) (\alpha_{-1} - \alpha_0)^2 h_{3,-2}
\\ &\hspace{20pt} + (\alpha_{-1} - \alpha_0)^2 h_{4,-2}
\\ &\hspace{20pt} + (\alpha_{-1} - 2 \alpha_0 + \alpha_3) h_{5,-2}
+ h_{6,-2}.
\end{align*}
\end{example}

\begin{example}
\label{ex:raising_operator}
Let us consider the raising operator formula
\begin{equation}
\label{eq:raising_operator_defn}
s_{\lambda}(\pp\dv\bal) = \prod_{1 \leqslant i < j \leqslant \ell} (1 - R_{ij}) h_{\lambda,\delta}.
\end{equation}
given in~\cite[Sec.~2]{Fun12} (noting a minor typo in the Jacobi--Trudi formula),\footnote{This is essentially expanding the determinant~\cite[9th Variation]{Macdonald92}.} where $R_{ij}$ is the \defn{raising operator} defined by
\[
R_{ij} h_{\lambda,\eta} = h_{\overline{\lambda},\overline{\eta}},
\text{ where } \overline{\lambda} = (\lambda_1, \dotsc, \lambda_{i-1}, \lambda_i+1, \lambda_{i+1}, \dotsc, \lambda_{j-1},\lambda_j-1,\lambda_{j+1}, \dotsc, \lambda_{\ell})
\]
and similarly for $\overline{\eta}$ from $\eta$.
Therefore, we can rewrite the leading factor in~\eqref{eq:raising_operator_integral} as, using the shorthand $(i|j)^k = (z_i^{-1}|\iota\sigma^j\bal)^k = (z_i^{-1} - \alpha_{-j}) \cdots (z_i^{-1} - \alpha_{-j-k+1})$,
\[
1 - \frac{(1|0)}{(2|0)} - \frac{(1|0)}{(3|1)} - \frac{(2|1)}{(3|1)}
+ \frac{(1|0)^2}{(2|0)(3|1)} + \frac{(1|0)}{(3|1)} + \frac{(1|0)(2|1)}{(3|0)^2}
- \frac{(1|0)^2}{(3|0)^2}
= \frac{\prod_{i < j} (z_j^{-1} - z_i^{-1})}{(2|0)(3|0)^2}
\]
In particular, we note that the sum does not naturally factor in the form of $\prod_{i<j} (1 - A'_{ij})$ nor can it be written as a determinant with entries $(i|j)^k$ over $(2|0)(3|0)^2$.
\end{example}

Even in the case $\bbe = 0$, we cannot easily answer~\cite[Prob.~5.32]{BHS}, which is to find a general (algebraic) proof of the identity in Example~\ref{ex:raising_operator}.
In addition, by comparing Example~\ref{ex:integral_raising} with Example~\ref{ex:raising_operator}, we see that even though the $\bal = 0$ case reduces to the classical integral formula proof of the raising operators (see, \textit{e.g.},~\cite[Eq.~(1.15)]{Baker96}), it does not extend to the $\bal$ case as one might expect.

\section{Skew-Pieri rule}
\label{sec:skew_pieri}

Next, we consider the $\bbe = 0$ version of the skew-Pieri rule from~\cite{BHS}.

\begin{corollary}[{Skew-Pieri formulas~\cite[Cor.~6.15]{BHS}}]
\label{cor:skew_pieri}
We have
\begin{subequations} \label{eq:skew-pieri}
\begin{align}
\label{eq:skew_pieri}
h_k(\pp'\dv\bal)s_{\mu/\nu}(\pp'\dv\bal) &= \sum_{\lambda,\eta} c_{k,\mu/\nu}^{\lambda/\eta}(\bal) s_{\lambda/\eta}(\pp'\dv\bal),
\\
\label{eq:dual_skew_pieri}
e_k(\pp'\dv\bal) s_{\mu/\nu}(\pp'\dv\bal) &= (-1)^k \sum_{\lambda,\eta} \overline{c}_{k,\mu/\nu}^{\lambda/\eta}(\bal) s_{\lambda/\eta}(\pp'\dv\bal),
\end{align}
\end{subequations}
where the sum in~\eqref{eq:skew_pieri} is over all partitions $\lambda,\eta$ such that $\lambda/\mu$ is a horizontal strip and $\nu/\eta$ is a vertical strip, the sum in~\eqref{eq:dual_skew_pieri} is over all partitions $\lambda,\eta$ such that $\lambda/\mu$ is a vertical strip and $\nu/\eta$ is a horizontal strip, and
\begin{subequations} \label{eq:skew-pieri-coeffs}
\begin{align}
c_{k,\mu/\nu}^{\lambda/\eta}(\bal) = \oint \widehat{s}_{\nu/\eta}(0/(-z)\dv\bal)  \widehat{s}_{\lambda/\mu}(z\dv\bal) \frac{(z; \bal)^{k-1}}{z^{k+1}} \frac{dz}{2\pi\ii},  \label{eq:skew_pieri_A}
\\
\overline{c}_{k,\mu/\nu}^{\lambda/\eta}(\bal) = \oint \widehat{s}_{\nu/\eta}(z\dv\bal) \widehat{s}_{\lambda/\mu}(0/(-z)\dv\bal) \frac{z^{-k-1}}{(z; \bal)^{1-k}} \frac{dz}{2\pi\ii}.  \label{eq:skew_pieri_B}
\end{align}
\end{subequations}
\end{corollary}

To compute them, we use~\cite[Cor.~6.13]{BHS}, which says $\widehat{s}_{\lambda/\mu}(\xx_1 \dv \bal) = \widehat{s}_{\nu/\eta}(1/\yy_1 \dv \bal)= 0$ unless $\lambda / \mu$ is a horizontal strip (resp.\ $\nu/\eta$ is a vertical strip).
Thus, we assume that $\lambda / \mu$ is a horizontal strip and $\nu/\eta$ is a vertical strip, and hence
\begin{subequations}
\label{eq:single_var}
\begin{align}
\widehat{s}_{\lambda/\mu}(\xx_1\dv\bal) & = \prod_{j=1}^{\ell} \frac{1 - \alpha_{j-\lambda'_j} x}{1 - \alpha_j x}  \prod_{\bbb\in \lambda/\mu} \frac{x}{1 - \alpha_{c(\bbb)} x},  \label{eq:single_x_var}
\\
\widehat{s}_{\nu/\eta}(0/\yy_1\dv\bal) & = \prod_{j=1}^{\ell} \frac{1 + \alpha_{\nu_j-j+1} y}{1 + \alpha_{1-j} y} \prod_{\bbb \in \nu/\eta} \frac{y}{1 + \alpha_{c(\bbb)+1} y}.  \label{eq:single_y_var}
\end{align}
\end{subequations}
Substituting in~\eqref{eq:single_var} into~\eqref{eq:skew_pieri_A}, we obtain
\begin{align*}
c_{k,\mu/\nu}^{\lambda/\eta}(\bal)
&= \oint \widehat{s}_{\lambda/\mu}(z\dv\bal)  \widehat{s}_{\nu/\eta}(0/(-z)\dv\bal) \frac{(z; \bal)^{k-1}}{z^{k+1}} \frac{dz}{2\pi\ii}
\\&= \oint \left( \prod_{j=1}^{\ell} \frac{1 - \alpha_{\nu_j-j+1} z}{1 - \alpha_{1-j} z} \prod_{\bbb \in \nu/\eta} \frac{-z}{1 - \alpha_{c(\bbb)+1} z}\right)
\\ & \hspace{30pt} \times \left(\prod_{j=1}^{\ell} \frac{1 - \alpha_{j-\lambda'_j} z}{1 - \alpha_j z}  \prod_{\bbb\in \lambda/\mu} \frac{z}{1 - \alpha_{c(\bbb)} z}\right) \frac{(z; \bal)^{k-1}}{z^{k+1}} \frac{dz}{2\pi\ii}.
\end{align*}
If we further impose the condition $\nu=\varnothing$, then $\eta=\varnothing$ as well.
Therefore, we have
\begin{equation}
\label{eq:Aklambda_integral}
c_{k\mu}^{\lambda}(\bal) := c_{k,\mu/\varnothing}^{\lambda/\varnothing}(\bal)
= \oint \left(\prod_{j=1}^{\ell} \frac{1 - \alpha_{j-\lambda'_j} z}{1 - \alpha_j z}  \prod_{\bbb\in \lambda/\mu} \frac{z}{1 - \alpha_{c(\bbb)} z}\right) \frac{(z; \bal)^{k-1}}{z^{k+1}} \frac{dz}{2\pi\ii}.
\end{equation}
We can see the valuation of the formal Laurent series is $r := k+1 - \abs{\lambda/\mu}$ (equivalently, the order of the pole at $z = 0$), and so the (formal contour) integral is $0$ whenever $\abs{\lambda/\mu} > k$.

Therefore, by standard symmetric function integral formulas, we have for any $\ell \geqslant \ell(\lambda)$
\begin{equation}
\label{eq:Aklambda_sf}
c_{k\mu}^{\lambda}(\bal)
= \sum_{s+t=k-|\lambda/\mu|} h_s(\{\alpha_{c(\bbb)} | \bbb\in\lambda/\mu\} \cup \bal_{[1,\ell]}) e_t(\{-\alpha_{j-\lambda_j'} \mid 1 \leqslant j \leqslant \ell\} \cup -\bal_{(0,k)}),
\end{equation}
where for a set $X = \{x_1, \dotsc, x_m\}$ we denote $f(X) = f(x_1, \dotsc, x_m)$ (since $f$ is a symmetric function, the order does not matter).
Strictly speaking, the case $k = 0$ is not evaluating the integral correctly with respect to the parameters, but since the sum only has the term $s = t = 0$, it yields the correct answer.

Similarly, consider the case where $\nu=\varnothing$; then $\eta=\varnothing$ as well and
\begin{align*}
\overline{c}_{k\mu}^{\lambda}(\bal)
&:= \overline{c}_{k,\mu/\varnothing}^{\lambda/\varnothing}(\bal)
= \oint \widehat{s}_{\lambda/\mu}(0/(-z)\dv\bal) \frac{z^{-k-1}}{(z; \bal)^{1-k}} \frac{dz}{2\pi\ii}
\\&= \oint \left(\prod_{j=1}^{\ell} \frac{1 - \alpha_{\lambda_j-j+1} z}{1 - \alpha_{1-j} z} \prod_{\bbb \in \lambda/\mu} \frac{-z}{1 - \alpha_{c(\bbb)+1} z}\right) \frac{z^{-k-1}}{(z; \bal)^{1-k}} \frac{dz}{2\pi\ii}
\\&= \oint \left(\prod_{j=1}^{\ell} \frac{1 - \alpha_{\lambda_j-j+1} z}{1 - \alpha_{1-j} z}  \prod_{\bbb \in \lambda/\mu} \frac{-z}{1 - \alpha_{c(\bbb)+1} z}\right) \frac{(z;\iota\bal)^{k-1}}{z^{k+1}} \frac{dz}{2\pi\ii}.
\end{align*}

By the same reasoning as above,
\begin{align*}
\overline{c}_{k\mu}^{\lambda}(\bal)
= (-1)^{|\lambda/\mu|} \sum_{s+t=k-|\lambda/\mu|} & h_s(\{\alpha_{c(\bbb)+1} | \bbb\in\lambda/\mu\} \cup \bal_{[1-\ell,0]})
\\ & \times e_t(\{-\alpha_{\lambda_j-j+1} \mid 1 \leqslant j \leqslant \ell\} \cup -\bal_{(1-k,1)}).
\end{align*}

Next, we compare to \cite[Prop.~3.4]{Fun12}. In the context of a single row, this is:

\begin{proposition}
\label{prop:fun_pieri}
The coefficient of the Pieri rule is given by
\[
c_{k\mu}^\lambda(\iota\bal) = \sum_T \prod_{\substack{b\in (k) \\ T(b) \text{ unbarred}}} (\alpha_{T(b) - \rho(b)_{T(b)}} - \alpha_{T(b)-c(b)}).
\]
where the sum is over all reverse $k$-supertableaux $T$ with row word sending $\mu$ to $\lambda$. Here, $T(b)$ is the value of $T$ at box $b$, $c(b)$ is the content of box $b$, and $\rho(b)$ is the partition obtained by adding all boxes up to $b$ to $\mu$.
\end{proposition}

The Yamanouchi condition in~\cite[Prop.~3.4]{Fun12} is equivalent to $\lambda / \mu$ being a horizontal strip.
A row word sending $\mu$ to $\lambda$ means that the barred entries of $T$ denote the rows of the boxes in $\lambda/\mu$. Therefore, $T$ is fixed, so the sum is a single term. In addition, $T(b) - \rho(b)_{T(b)}$ is just the negative content of the added box in $\lambda/\mu$, while $T(b)-c(b)$ is the row of $\lambda/\mu$ of a given box, minus $i-1$ if this is the $i$th box of $\lambda/\mu$ from left to right.

Let us note that our formula \eqref{eq:Aklambda_sf} does not immediately imply Graham positivity, which means the coefficients belong to $\ZZ_{\geqslant 0}[\alpha_{i-1} - \alpha_i \mid i \in \ZZ]$.
Yet, we can easily see a shadow of it by setting $\bal = \alpha$, where the contour integral~\eqref{eq:Aklambda_integral} becomes $0$ unless $\abs{\lambda/\mu} = k$ by simply counting the factors in the numerator and denominator (as per Remark~\ref{rem:substitution}). 
On the other hand, our formulas are ``compressed'' as a sum of monomials, in the sense that if we remove common factors from the numerator and denominator of the contour integrals~\eqref{eq:Aklambda_integral} (equivalently, occurring in both the elementary and homogeneous symmetric function inputs in~\eqref{eq:Aklambda_sf}), then we get no cancellations.

\begin{example}
Consider $\mu = (4, 3)$, $k = 4$, and $\lambda = (4,4,1)$.
Then from Proposition~\ref{prop:fun_pieri}, the corresponding coefficient is given as the sum over
\[
\newcommand{\bth}{\overline{3}}%
\newcommand{\btw}{\overline{2}}%
\ytableausetup{boxsize=1.3em}%
\begin{array}{ccc}
\ytableaushort{{\bth}{\btw}22} & \ytableaushort{{\bth}{\btw}21} & \ytableaushort{{\bth}{\btw}11}
\\[3pt]
(\alpha_{-2} - \alpha_0)(\alpha_{-2} - \alpha_{-1}) & (\alpha_{-2} - \alpha_0)(\alpha_{-3} - \alpha_{-2}) & (\alpha_{-3} - \alpha_{-1})(\alpha_{-3} - \alpha_{-2})
\\[3pt]
\ytableaushort{{\bth}2{\btw}2} &\ytableaushort{{\bth}2{\btw}1} & \ytableaushort{{\bth}22{\btw}}
\\[3pt]
(\alpha_{-1} - \alpha_1)(\alpha_{-2} - \alpha_{-1}) & (\alpha_{-1} - \alpha_1)(\alpha_{-3} - \alpha_{-2}) & (\alpha_{-1} - \alpha_1)(\alpha_{-1} - \alpha_{0}).
\end{array}
\]
Applying $\iota$ and summing the result yields $\alpha_0 \alpha_1 - \alpha_0 \alpha_4 - \alpha_1 \alpha_4 + \alpha_4^2$.
Now noting $\lambda' = (3,2,2,2)$ and taking $\ell = 4$, we have that~\eqref{eq:Aklambda_sf} yields
\begin{align*}
c_{k,\mu}^{\lambda} & = \sum_{s+t=2} h_s(\{\alpha_{-2},\alpha_2 \} \cup \{ \alpha_1, \alpha_2, \alpha_3, \alpha_4\}) e_t(\{-\alpha_{-2}, -\alpha_0, -\alpha_1, -\alpha_2\} \cup \{ -\alpha_1, -\alpha_2, -\alpha_3\})
\\
& = \sum_{s=0}^2 h_s(\{\alpha_{4} \}) e_{2-s}(\{-\alpha_0, -\alpha_1\})
\\ & = \alpha_0 \alpha_1 - (\alpha_0 + \alpha_1) \alpha_4 + \alpha_4^2 = (\alpha_4 - \alpha_0)(\alpha_4 - \alpha_1).
\end{align*}
\end{example}

\begin{example}
We note that in general, there is no such nice factorization of the Pieri rule coefficients:
\begin{align*}
h_2 s_{522} &= s_{722} + s_{632} + s_{6221} + s_{542} + s_{5321} + s_{5222}
\\ & \hspace{20pt} + (\alpha_5 + \alpha_6 - \alpha_{-1} - \alpha_{-2}) s_{622} + (\alpha_2 + \alpha_5 - \alpha_{-1} - \alpha_{-2}) s_{532} + (\alpha_5 - \alpha_{-1}) s_{5221}
\\ & \hspace{20pt} + (\alpha_5 - \alpha_{-1})(\alpha_5 - \alpha_{-2}) s_{522}.
\end{align*}
\end{example}

\bibliographystyle{habbrv}
\bibliography{focktorial}

\end{document}